\documentclass[portrait,reqno]{amsart}%
\usepackage{amsfonts}
\usepackage{graphicx}
\usepackage{amsmath}
\usepackage{amssymb}
\usepackage[latin1]{inputenc}
\usepackage{a4}%
\providecommand{\U}[1]{\protect\rule{.1in}{.1in}}
\newtheorem{thm}{Theorem}[section]

\newtheorem{lem}[thm]{Lemma}
\newtheorem{prop}[thm]{Proposition}
\theoremstyle{definition}
\newtheorem{defn}[thm]{Definition}
\theoremstyle{remark}
\newtheorem{rem}[thm]{Remark}

\numberwithin{equation}{section}
\newtheorem{ex}[thm]{Example}

\begin{document}
\title[$\chi^{2}$ estimation with application to test of hypotheses]{An estimation method for the chi-square divergence with application to test of hypotheses}
\author{M. Broniatowski$^{1}$}
\author{S. Leorato$^{2}$}
\address{$^{1}$Université de REIMS and LSTA, Université PARIS 6, 175 Rue du Chevaleret
75013 PARIS, FRANCE \\
$^{2}$Dip. di Statistica, Probabilitá e Statistiche Applicate, University of
Rome "La Sapienza", P.le A. Moro, 5 00185 Roma}
\email{$^{2}$samantha.leorato@uniroma1.it}
\keywords{chi-square divergence, hypothesis testing, linear constraints, marginal
distributions, contamination models, Fenchel-Legendre transform, inliers}
\subjclass{62F03, 62F10,62F30}
\maketitle

\begin{abstract}
We propose a new definition of the chi-square divergence between
distributions. Based on convexity properties and duality, this version of the
$\chi^{2}$ is well suited both for the classical applications of the $\chi
^{2}$ for the analysis of contingency tables and for the statistical tests for
parametric models, for which it has been advocated to be robust against
\emph{inliers}.

We present two applications in testing. In the first one we deal with tests
for finite and infinite numbers of linear constraints, while, in the second
one, we apply $\chi^{2}-$methodology for parametric testing against contamination.

\end{abstract}




\section{Introduction}

\label{sec:1.intro} The $\chi^{2}$ distance is commonly used for categorized
data. For the continuous case, optimal grouping pertaining to the $\chi^{2}$
criterion have been proposed by various authors; see f.i. {~\cite{Bosq80}},
\cite{Lancaster}, \cite{greenwoodNikulin}. These methods are mainly applied
for tests, since they may lead to some bias effect for estimation.

This paper introduces a new approach to the $\chi^{2}$, inserting its study
inside the range of divergence-based methods and presenting a technique
avoiding grouping for estimation and test. Let us first introduce some notation.

Let $M_{1}$ denote the set of all probability measures on $\mathbb{R}^{d}$ and
$M $ the set of all signed measures on $\mathbb{R}^{d}$ with total mass 1. For
$P \in M_{1}$ and $Q \in M$, introduce the $\chi^{2}$ distance between $P$ and
$Q$ by%

\begin{equation}
\chi^{2}(Q,P)=\left\{
\begin{array}
[c]{ll}%
\int_{{}}\left(  \frac{dQ-dP}{dP}\right)  ^{2}\,dP &
Q\mbox{ is a.c. w.r.t. }P\\
\infty & \mbox{otherwise.}
\end{array}
\right. \label{eqn:1.def_chi2}%
\end{equation}

For $\Omega$ a subset of $M$ denote
\begin{equation}
\chi^{2}(\Omega,P)=\inf_{Q\in\Omega}\chi^{2}(Q,P),\label{eqn:1.def_chi2_Omega}%
\end{equation}
with $\inf_{\{\emptyset\}}=\infty$.

When the infimum in (\ref{eqn:1.def_chi2_Omega}) is reached at some measure
$Q^{*}$ which belongs to $\Omega$, then $Q^{*}$ is the \emph{projection} of
$P$ to $\Omega$. Also the role of the class of measures $M$ will appear later,
in connection with the possibility to obtain easily $Q^{*}$ through usual
optimization methods, which might be quite difficult when we consider subsets
$\Omega$ in $M_{1}$.

For a problem of test such as $H_{0}:P\in\Omega$ vs $H_{1}:P\not \in \Omega$,
the test statistic will be an estimate of $\chi^{2}(\Omega,P)$, which equals
$0$ under $H_{0}$, since in that case $P=Q^{\ast}$. Therefore, under $H_{0}$,
there is no restriction when considering $\Omega$ a subset of $M$.

The $\chi^{2}$ distance belongs to the so-called $\phi-$divergences, defined
through
\begin{equation}
\phi(Q,P)=\left\{
\begin{array}
[c]{ll}%
\int{\varphi\left(  \frac{dQ}{dP}\right)  }dP & \mbox{when }Q\ll P\\
\infty & \mbox{otherwise}
\end{array}
\right. \label{eqn:1.def_phi_div}%
\end{equation}
where $\varphi$ is a convex function defined on $\mathbb{R}^{+}$ satisfying
$\varphi(1)=0$. This class of discrepancy measures between probability
measures has been introduced by I. Csisz\'ar \cite{Csiszar1967}, and the
monograph by F. Liese and I. Vajda \cite{LV77} provides their main properties.

The extension of $\phi-$divergences when $Q$ is assumed to be in $M$ is
presented in \cite{Broniatowski-Keziou2003}, in the context of parametric
estimation and tests.

The class of minimum $\phi-$divergence test statistics include, within the
others, the \emph{loglikelihood ratio test}.

For this class it is a matter of fact that first order efficiency is not a
useful criterion of discrimination. A notion of robustness against model
contamination is found in Lindsay \cite{Lindsay1994} (for estimators) and in
Jimenez and Shao \cite{JimenezShao2001} (for test procedures), which gives an
instrument to compare the tests associated to different divergences. Although
their argument deals with finite support models, it may help as a benchmark
for more general situations.

By these papers it emerges that the \emph{minimum Hellinger distance} test
provides a reasonable compromise between robustness against model
contaminations induced by outliers and by inliers.

However, when the model might be subject to inlier contaminations only (namely
missing data problems), as will be advocated in the present paper for
contamination models, then minimum $\chi^{2}-$divergence test behaves better
than both minimum Hellinger distance and loglikelihood ratio tests, in terms
of their \emph{residual adjustment functions} (RAF), because (we refer to
\cite{JimenezShao2001} for the notation)
\[
\left| \frac{A_{\chi^{2}}(-1)}{A_{LR}(-1)}\right| = \frac{1}{2}<1
\quad\mbox{and} \qquad\left| \frac{A_{\chi^{2}}(-1)}{A_{HD}(-1)}\right|
=\frac{1}{4}<1.
\]

Formula (\ref{eqn:1.def_chi2}) is not suitable for statistical purposes as
such. Indeed, suppose that we are interested in testing wether $P$ is in some
class $\Omega$ of distributions with absolutely continuous component. Let
$X=(X_{1},\ldots,X_{n})$ be an i.i.d. sample with unknown distribution $P$.
Assume that $P_{n}:=\frac{1}{n}\sum_{i=1}^{n}\delta_{X_{i}}$, the empirical
measure pertaining to $X$, is the only information available on $P$, where
$\delta_{x}$ is the Dirac measure at point $x$. Then, for all $Q\in\Omega$,
the $\chi^{2}$ distance between $Q$ and $P_{n}$ is infinite. Therefore no
plug-in technique can lead to a definite statistic in this usual case.

Our approach solves this difficulty and is based on the \textquotedblright
dual representation \textquotedblright\ for the $\chi^{2}$ divergence, which
is a consequence of the convexity of the mapping $Q\longmapsto\chi^{2}(Q,P)$,
plus some regularity property; this will be set in Section
\ref{sec:2.estimator}, together with conditions under which $P$ has a
$\chi^{2}- $projection on $\Omega$. We will also provide an estimate for the
function $\frac{dQ^{\ast}}{dP}$ which indicates the local changes induced on
$P$ by the projection operator.

In some cases it is possible to replace $\Omega$ by $\Omega\cap\Lambda_{n}$
\ where $\Lambda_{n}$ is the set of all measures in $M$ whose support is $X$,
when this intersection is not void, as happens when $\Omega$ is defined for
example through moment conditions. This approach is called the Generalized
Likelihood paradigm (see \cite{neweySmith2003} and references therein), and we
will develop in Section 3 a complete study pertaining to such case when
handling the $\chi^{2}$ divergence, in the event that $\Omega$ is defined
through linear constraints, namely when
\begin{equation}
\Omega=\left\{  Q\in M\text{such that}\int f(x)dQ(x)=0\right\}
\label{eqn:1.omegalinear}%
\end{equation}
for some $\mathbb{R}^{k}-$valued function $f$ defined on $\mathbb{R}^{d}$. In
this case the projection $Q^{\ast}$ has a very simple form and its estimation
results as the solution of a linear system of equations, which motivates the
choice of $\chi^{2}$ criterion for tests of the form $H_{0}:P\in\Omega$ with
$\Omega$ as in (\ref{eqn:1.omegalinear}). As is shown in Section 3, by Theorem
\ref{th:2.linear_char} the constrained problem is in fact reduced to an
unconstrained one.

Also for the problem of testing whether $P$ belongs to $\Omega$ our results
include the asymptotic distribution of the test statistics under any $P$ in
the alternative, proving consistency of the procedure, a result that is not
addressed in the current literature on Generalized Likelihood.

In Section 3 we will apply the above results to the case of a test of fit,
where $\Omega=\left\{  P_{0}\right\}  $ is a fixed p.m.

When $\Omega\cap\Lambda_{n}$ is void some smoothing technique has been
proposed, following \cite{Beran77}, substituting $P_{n}$ by some regularized
version; see \cite{MoralesPardoVajda1997}. In those cases we have chosen not
to make any smoothing, exploiting the dual representation in a parametric
context. Section 4 addresses this approach through the study of contamination
models, for a composite problem, when the contamination modifies a
distribution with unknown parameter.

\section{The definition of the estimator}

\label{sec:2.estimator}

\subsection{Some properties of $\chi^{2}-$distance}

\label{subsec:2.duality}

We will consider sets $\Omega$ of signed measures with total mass 1 that
integrate some class of functions $\Phi$. The choice of $\Phi$ depends on the
context as seen below. Let
\begin{equation}
M_{\Phi}:=\left\{  Q\in M\text{ \ such that \ }\int|\varphi|\text{
}d|Q|<\infty,\mbox{ for
all }\varphi\in\Phi\right\}  .\label{eqn:2.M_Phi}%
\end{equation}
We first consider sufficient conditions for the existence of $Q^{\ast}$, the
projection of $P$ on $\Omega.$ We introduce the following notation.

Let ${\boldsymbol{\Phi}}=\Phi\cup\mathcal{B}_{b}$, where $\mathcal{B}_{b}$ is
the class of all measurable bounded functions on $\mathbb{R}^{d}$. Let
$\tau_{{\boldsymbol{\Phi}}}$ be the coarsest topology on $M$ which makes all
mappings $Q\longmapsto\int\varphi dQ$ continuous for all $\varphi
\in{\boldsymbol{\Phi}} $. When $\boldsymbol{\Phi} $ is restricted to
$\mathcal{B}_{b} $, the $\tau_{{\boldsymbol{\Phi}}}$ topology turns out to be
the usual $\tau- $ topology (see e.g. \cite{GOR1979}).

Assume that for all functions $\varphi$ in $\Phi$ there exists some positive
$\varepsilon$ with
\[
\int\varphi^{2+\varepsilon}dP<\infty.
\]
Whenever $\Omega$ is a closed set in $M_{\Phi}$ equipped with the
$\tau_{\boldsymbol{\Phi}} $ topology and $\chi^{2}(\Omega,P)$ is finite, then,
as a consequence of Theorem 2.3 in \cite{Broniatowski-Keziou2003}, $P$ has a
projection in $\Omega$. Moreover, when $\Omega$ is convex, uniqueness is achieved.

In statistical applications the set $\Omega$ is often defined through some
statistical functional; for example, let $\Omega$ defined as in
(\ref{eqn:1.omegalinear}). In this case $\Phi:=\left\{  f\right\} $ and
$\Omega$ is closed by the very definition of $\Phi$; therefore the choice of
the class of functions $\Phi$ is intimately connected with the set $\Omega$
under consideration. As seen in Section 4, and as developed in
\cite{Broniatowski-Keziou2003} also when $\Omega$ is a subset of some
parametric family of distributions, the class $\Phi$ can be defined with
respect to $\Omega$.

We first provide a characterization of the $\chi^{2}-$projection of a p.m. $P$
on some set $\Omega$ in $M$.

Let $\mathcal{D}$ denote the domain of the divergence for fixed $P$, namely
\[
\mathcal{D}=\left\{  Q\in M\text{ such that \ }\chi^{2}(Q,P)<\infty\right\}  .
\]

We have (see \cite{BroniatowskiKeziou2003b}, Theorem 2.6)

\begin{thm}
\label{th:2.characterization} Let $\Omega$ be a subset of $M$. Then

\begin{enumerate}
\item If there exists some $Q^{\ast}$ in $\Omega$ such that for all $Q$ in
$\Omega\cap\mathcal{D}$,
\[
q^{\ast}\in L_{1}(Q)\mbox{ and }\int q^{\ast}dQ^{\ast}\leq\int q^{\ast}dQ
\]
where $q^{\ast}=\frac{dQ^{\ast}}{dP}$, then $Q^{\ast}$ is the $\chi^{2}%
-$projection of $P$ on $\Omega$

\item If $\Omega$ is convex and $P$ has projection $Q^{*}$ on $\Omega$, then,
for all $Q$ in $\Omega$, $q^{*}$ belongs to $L_{1}(P)$ and $\int q^{*}
dQ^{*}\leq\int q^{*} dQ $.
\end{enumerate}
\end{thm}

Many statistically relevant problems in estimation and testing pertain to
models defined by linear constraints (Empirical Likelihood paradigm and
others). Section 3 is devoted to this case. We therefore present a
characterization result for the $\chi^{2}-$projection on sets of measures
defined by linear constraints.

Let $\Phi$ be a collection (finite or infinite, countable or not) of real
valued functions defined on $\mathbb{R}^{d}$, which we assume to contain the
function 1. Let $\Omega$ a subset of $M$ be defined by
\[
\Omega=\left\{  Q\in M\text{ \ such that \ }\int gdQ=0\;\text{\ for all
\ }\;g\text{ in }\Phi-\{1\}\right\}  .
\]
Denote $<\Phi>$ the linear span of $\Phi$.

We then have the following result (see \cite{BroniatowskiKeziou2003b}):

\begin{thm}
\label{th:2.linear_char}

\begin{enumerate}
\item $P$ has a projection $Q^{\ast}$ in $\Omega$ iff $Q^{\ast}$ belongs to
$\Omega$ and for all $Q\in\Omega$, $q^{\ast}\in L_{1}(Q)$ and $\int q^{\ast
}dQ^{\ast}\leq\int qdQ^{\ast}$.

\item If $q^{\ast}$ belongs to $<\Phi>$ and $Q^{\ast}$ belongs to $\Omega$,
then $Q^{\ast}$ is the projection of $P$ on $\Omega$.

\item If $P$ has projection $Q^{*}$ on $\Omega$, the $q^{*}$ belongs to
$\overline{<\Phi>}$, the closure of $\Phi$ in $L_{1}(Q^{*})$.
\end{enumerate}
\end{thm}

\begin{rem}
\label{rem:2.counterexample} The above result only provides a partial answer
to the characterization of the projections. Let $P$ be the uniform
distribution on $[0,1]$. The set $M_{1}(P)$ of all p.m.'s absolutely
continuous with respect to $P$ is a closed subset of $M_{\Phi}$, when
$\Phi:=\left\{  x\mapsto x\right\}  \cup\left\{  x\mapsto1\right\}  $. Let
$\Omega:=\left\{  Q\in M_{1}(P):\;\int xdQ(x)=\frac{1}{4}\right\}  $. Then $P$
has a projection on $\Omega$ and $\frac{dQ^{\ast}}{dP}(x){}{\Large 1}%
_{\left\{  q^{\ast}>0\right\}  }(x)=c_{0}+c_{1}x$, with $q^{\ast}%
=\frac{dQ^{\ast}}{dP} $. The support of $Q^{\ast}$ is strictly included in
$[0,1]$. Otherwise we obtain $c_{0}=\frac{5}{2}$ and $c_{1}=-3$, a
contradiction, since then $Q^{\ast}$ is not a probability measure.
\end{rem}

\subsection{An alternative version of the $\chi^{2}$}

\label{subsec:3.duality} The $\chi^{2}$ distance defined on $M$ for fixed $P $
in $M_{1}$ through $\chi^{2}(Q,P)=\int\left(  \frac{dQ}{dP}-1\right)  ^{2}dP$
is a convex function; as such it is the upper envelope of its support
hyperplanes. The first result, which is Proposition 2.1 in
\cite{Broniatowski-Keziou2003}, provides the description of the hyperplanes in
$M_{\Phi}$.

\begin{prop}
\label{th:2.hausdorff} Equip $M_{\Phi}$ with the $\tau_{\boldsymbol{\Phi}}%
-$topology. Then $M_{\Phi}$ is a Hausdorff locally convex topological space.
Further, the topological dual space of $M_{\Phi}$ is the set of all mappings
$Q\mapsto\int fdQ$ when $f$ belongs to $<{\boldsymbol{\Phi}} >$.
\end{prop}

Proposition 2.3 in \cite{Broniatowski-Keziou2003} asserts that the $\chi^{2} $
distance defined on $M$ for fixed $P$ in $M_{1}$ is l.s.c. in $\left(
M_{\Phi},\tau_{{\boldsymbol{\Phi}} }\right)  $. We can now state the duality lemma.

Define on $<\boldsymbol{\Phi} >$, the Fenchel-Legendre transform of $\chi
^{2}(\cdot,P)$
\begin{equation}
\label{eqn:T(f,P)}T(f,P):=\sup_{Q\in M_{\Phi}}\int fdQ-\chi^{2}(Q,P).
\end{equation}

We have

\begin{lem}
\label{th:2.dualitylemma} The function $Q\longmapsto\chi^{2}(Q,P)$ admits the
representation
\begin{equation}
\chi^{2}(Q,P)=\sup_{f\in<{\boldsymbol{\Phi}}>}\int
fdQ-T(f,P).\label{eqn:2.dualrepres}%
\end{equation}

\end{lem}

Standard optimization techniques yield
\[
T(f,P)=\int fdP+\frac{1}{4}\int f^{2}dP
\]
for all $f\in<{\boldsymbol{\Phi}} >$, see e.g. \cite{Aze1997}, Chapter 4.

The function $f^{\ast}=2\left(  \frac{dQ^{\ast}}{dP}-1\right)  $ is the
supremum in (\ref{eqn:2.dualrepres}) as can be seen through classical convex
optimization procedures.

We now consider a subclass $\mathcal{F}$ in $<{\boldsymbol{\Phi}}> $ and we assume:

\begin{itemize}
\item[(C1)] $f^{*}$ belongs to $\mathcal{F}$. \label{(C1)}
\end{itemize}

Therefore
\[
\chi^{2}(Q,P)=\sup_{f\in\mathcal{F}}\int fdQ-T(f,P)
\]
which we call the \emph{dual representation} of the $\chi^{2}$.

This can be restated as follows: let
\[
m_{f}(x):= \int f dQ - \left( f(x)+\frac{1}{4} f^{2}(x)\right) .
\]
Then
\begin{equation}
\label{eqn:2.chi2dual}\chi^{2}(Q,P)=\sup_{f \in\mathcal{F}}\int m_{f}(x)dP(x).
\end{equation}

Hence we have
\begin{equation}
\label{eqn:2.chi2Omegadual}\chi^{2}(\Omega,P)=\inf_{Q \in\Omega}\sup_{f
\in\mathcal{F}}\int m_{f}(x)dP(x).
\end{equation}

In the case when $\Omega$ is defined through a finite number of linear
constraints, say
\[
\Omega=\left\{  Q\in M:\;\int f_{i}(x)dQ(x)=a_{i},\;1\leq i\leq k\right\}  ,
\]
when $P$ has a projection $Q^{\ast}$ on $\Omega$ and $supp\{Q^{\ast}\}$ is
known to coincide with that of $P$, then we may choose $\mathcal{F}$ as the
linear span of $\{1,f_{1},\ldots,f_{k}\}$ and (\ref{eqn:2.chi2Omegadual})
turns out to be a parametric unconstrained optimization problem, since, by
Theorem \ref{th:2.linear_char} (3)
\[
\chi^{2}(\Omega,P)=\sup_{c_{0},c_{1},\ldots,c_{k}}c_{0}+\sum_{i=1}^{k}%
c_{i}a_{i}-T\left(  c_{0}+\sum_{i=1}^{k}c_{i}f_{i},P\right)  .
\]

In some other cases we may have a complete description of all functions
$\frac{dQ}{dP}$ when $Q$ belongs to $\Omega$. A typical example is when $P$
and $Q$ belong to parametric families.

\subsection{The estimator $\chi_{n}^{2}$}

\label{subsec:2.estimator}

Let us now present the estimate of $\chi^{2}(\Omega,P)$.

Together with an i.i.d. sample $X_{1},\ldots,X_{n}$ with common unknown
distribution $P$, define the estimate of $\chi^{2}(Q,P)$ through
\begin{equation}
\chi_{n}^{2}(Q,P):=\sup_{f\in\mathcal{F}}\int m_{f}(x)dP_{n}%
(x)\label{eqn:2.def_chi_n}%
\end{equation}
a plug-in version of (\ref{eqn:2.chi2dual}).

We also define the estimate of $\chi^{2}(\Omega,P)$ through
\begin{equation}
\chi_{n}^{2}(\Omega,P):=\inf_{Q\in\Omega}\sup_{f\in\mathcal{F}}\int
m_{f}(x)dP_{n}(x).\label{estim chi2}%
\end{equation}

These estimates may seem cumbersome. However, in the case when we are able to
reduce the class $\mathcal{F}$ to a reasonable degree of complexity, these
estimates perform quite well and can be used for testing $P\in\Omega$ against
$P\not \in \Omega$. This will be made clear in the last two sections which
serve as examples for the present approach.

In some cases it is possible to commute the $\sup$ and the $\inf$ operators in
(\ref{eqn:2.chi2Omegadual}), which turns out to become%

\begin{equation}
\chi^{2}(\Omega,P)=\sup_{f\in\mathcal{F}}\inf_{Q\in\Omega}\int
fdQ-T(f,P),\label{eqn:2.chi2_minimax}%
\end{equation}
in which the $\inf$ operator acts only on the linear functional $\int fdQ$.


Also, when (\ref{eqn:2.chi2_minimax}) holds, we may define an estimate of
$\chi^{2}(\Omega,P)$ through
\begin{equation}
\label{eqn:2.chi_n_minimax}\overline{\chi}_{n}^{2}(\Omega,P)=\sup_{f
\in\mathcal{F}}\inf_{Q \in\Omega} \int f dQ- T(f,P_{n}).
\end{equation}

When (\ref{eqn:2.chi2_minimax}) holds, it is quite easy to get the limit
properties of $\overline{\chi}_{n}^{2}$.

Indeed, by (\ref{eqn:2.chi2_minimax}) and (\ref{eqn:2.chi_n_minimax})%

\begin{align*}
\overline{\chi}_{n}^{2}(\Omega,P)-\chi^{2}(\Omega,P) & = \left( \sup_{f
\in\mathcal{F}}\inf_{Q \in\Omega}\int f dQ- T(f,P_{n})\right) -\left( \sup_{f
\in\mathcal{F}}\inf_{Q \in\Omega}\int f dQ- T(f,P)\right) .
\end{align*}

Now define
\[
\phi_{R}(f):=\inf_{Q\in\Omega}\int fdQ-T(f,R)=\inf_{Q\in\Omega}\int
fdQ-\int\left(  f+\frac{1}{4}f^{2}\right)  dR,
\]
a concave function of $f$.

When $\mathcal{F}$ is compact in a topology for which $\phi_{R}$ is uniformly
continuous for all $R$ in $M_{1}$, then a sufficient condition for the a.s.
convergence of $\overline{\chi}_{n}^{2}(\Omega,P)$ to $\chi^{2}(\Omega,P)$ is
\[
\lim_{n\rightarrow\infty}\sup_{f\in\mathcal{F}}\left\vert \phi_{P_{n}}%
(f)-\phi_{P}(f)\right\vert =0\;\;a.s.
\]
which in turn is
\[
\lim_{n\rightarrow\infty}\sup_{f\in\mathcal{F}}\left\vert \int\left(
f+\frac{1}{4}f^{2}\right)  dP_{n}-\int\left(  f+\frac{1}{4}f^{2}\right)
dP\right\vert =0\;\;a.s.
\]

This clearly holds when the class of functions $\left\{  \left(  f+\frac{1}%
{4}f^{2}\right)  ,\;f\in\mathcal{F}\right\}  $ satisfies the functional
Glivenko-Cantelli (GC) condition (see \cite{Pollard84}).

The limit distribution of the statistic $\overline{\chi}_{n}^{2}(\Omega,P)$
under $H1$, i.e. when $P$ does not belong to $\Omega$, can be obtained under
the following hypotheses, following closely the proof of Theorem 3.6 in
\cite{Broniatowski2002}, where a similar result is proved for the
Kullback-Leibler divergence estimate.

Assume

\begin{itemize}
\item[(C2)] \label{(C2)} $P$ has a unique projection $Q^{\ast}$ on $\Omega$.

\item[(C3)] \label{(C3)} The class $\mathcal{F}$ is compact in the sup-norm.

\item[(C4)] \label{(C4)} The class $\left\{ f+\frac{1}{4}f^{2}, \; f
\in\mathcal{F}\right\} $ is a functional Donsker class.
\end{itemize}

We then have

\begin{thm}
\label{th:2.weakconv_H1} Under $H1$, assume that \emph{(C1)}--\emph{(C4)}
hold. The asymptotic distribution of
\[
\sqrt{n}\left( \overline{\chi}_{n}^{2}(\Omega,P)-\chi^{2}(\Omega,P)\right)
\]
is that of $B_{P}(g^{*})$, where $B_{P}(\cdot)$ is the $P-$Brownian bridge
defined on $\mathcal{F}$, and $g^{*}=-f^{*}-\frac{1}{4}{f^{*}}^{2}$.
\end{thm}

Therefore $\sqrt{n}\left(  \overline{\chi}_{n}^{2}(\Omega,P)-\chi^{2}%
(\Omega,P)\right)  $ has an asymptotic centered normal distribution with
variance $E_{P}\left(  \left(  f^{\ast}+\frac{1}{4}{f^{\ast}}^{2}\right)
^{2}(X)\right)  -E_{P}\left(  \left(  -f^{\ast}-\frac{1}{4}{f^{\ast}}%
^{2}\right)  (X)\right)  ^{2}$, where $X$ has law $P$.

The asymptotic distribution of $\overline{\chi}_{n}^{2}$ under $H0$, i.e. when
$P $ belongs to $\Omega$, cannot be obtained in a general frame and must be
derived accordingly to the context.

In the next Sections we develop two applications of the above statements. In
the first one we consider sets $\Omega$ defined by an infinite number of
linear constraints. We approximate $\Omega$ through some sieve technique and
provide consistent test for $H_{0}:P\in\Omega$. We specialize this problem to
the two sample test for paired data. So, in this first application, we
basically use the representation of the projection $Q^{\ast} $ of $P$ on
linear sets as described through Theorem \ref{th:2.linear_char}. In this first
range of applications we will project $P_{n}$ on the non void set $\Omega
\cap\Lambda_{n}$.

The second application deals with parametric models and test for
contamination. We obtain a consistent test for the case when $\Omega$ is a set
of parametrized distributions $F_{\theta}$ for $\theta$ in $\Theta
\subset\mathbb{R}^{d}$. The test is
\[
H0:\; P \in\Omega=\left\{  F_{\theta},\: \theta\in\Theta\right\} ,\:
\mbox{i.e.} \; \lambda=0 \newline\quad\mbox{vs}\quad H1: \:P\in\{(1-\lambda
)F_{\theta}+\lambda R,\, \lambda\not = 0,\, \theta\in\Theta\}.
\]
In this example we project $P_{n}$ on a set of absolutely continuous
distributions and we make use of the minimax assumption
(\ref{eqn:2.chi2_minimax}) which we prove to hold.

\section{Test of a set of linear constraints}

\label{sec:3.linear} Let $\mathcal{F}$ be a countable family of real-valued
functions defined on $\mathbb{R}^{d}$, $\{a_{i}\}_{i=1}^{\infty}$ a real
sequence and%

\begin{equation}
\Omega:=\left\{  Q\in M\text{ such that \ }\int f_{i}dQ(x)=a_{i}%
,\;i\geq1\right\} \label{eqn:3.omega_linear}%
\end{equation}
We assume that $\Omega$ is not void. In accordance with the previous section
we assume that the function $f_{0}:=1$ belongs to $\mathcal{F}$ with $a_{0}=1$.

Let $X_{1},\ldots,X_{n}$ be an i.i.d. sample with common distribution $P$.

We intend to propose a test for $H_{0}:P \in\Omega$ vs $H_{1}: P
\not \in \Omega$.

We first consider the case when $\mathcal{F}$ is a finite collection of
functions, and next extend our results to the infinite case.

For notational convenience we write $Pf$ for $\int fdP$ whenever defined.

\subsection{Finite number of linear constraints}

\label{subsec:3.finite} Consider the set $\Omega$ defined in
(\ref{eqn:3.omega_linear}) with $card\{\mathcal{F}\}=k$. Introduce the
estimate of $\chi^{2}(\Omega,P)$ through
\begin{equation}
\chi_{n}^{2}(\Omega,P)=\inf_{Q\in\Omega\cap\Lambda_{n}}\chi^{2}(Q,P_{n}%
).\label{estim lin}%
\end{equation}
Embedding the projection device in $M\cap\Lambda_{n}$ instead of $M_{1}%
\cap\Lambda_{n}$ yields to a simple solution for the optimum in
(\ref{estim lin}), since no inequality constrains will be used. Also the
topological context is simpler than as mentioned in the previous section since
the projection of $P_{n}$ belongs to $\mathbb{R}^{n}$. When developed in
$M_{1}\cap\Lambda_{n}$ this approach is known as the Generalized Likelihood
(GEL) paradigm (see \cite{MR2002302}). Our approach differs from the lattest
through the use of the dual representation (\ref{estim chi2}), which provides
consistency of the test procedure.


It is readily checked that
\[
\chi_{n}^{2}(\Omega,P)=\overline{\chi}_{n}^{2}(\Omega,P)
\]
The set $\Omega\cap\Lambda_{n}$ is a convex closed subset in $\mathbb{R}^{n}$.
When the projection of $P_{n}$ on $\Omega\cap\Lambda_{n}$ exists uniqueness
therefore holds. In the next section we develop various properties of our
estimates, which are based on the duality formula (\ref{eqn:2.def_chi_n}).

The next subsections provide all limit properties of $\chi_{n}^{2}(\Omega,P)$.

\subsubsection{Notation and basic properties}

Let $Q_{0}$ be any fixed measure in $\Omega$. By (\ref{eqn:2.chi2Omegadual})
\begin{align}
\chi^{2}\left(  \Omega,P\right)   &  =\sup_{f\in< \mathcal{F}> }\left(
Q_{0}-P\right)  f-\frac{1}{4}Pf^{2}\nonumber\\
&  =\sup_{a_{0},a_{1,}\ldots,a_{k}}\sum_{i=1}^{k}a_{i}\left(  Q_{0}-P\right)
f_{i}-\frac{1}{4}P\left(  \sum_{i=1}^{k}a_{i}f_{i}+a_{0}\right)
^{2}.\label{linear-chi_n}%
\end{align}
since, for $Q$ in $\Omega$ and for all $f$ in $\mathcal{F}$, $Qf=Q_{0}f$ and
\[
\chi^{2}_{n}=\sup_{a_{0},a_{1,}\ldots,a_{k}}\sum_{i=1}^{k}a_{i}\left(
Q_{0}-P_{n}\right)  f_{i}-\frac{1}{4}P_{n}\left(  \sum_{i=1}^{k}a_{i}%
f_{i}+a_{0}\right)  ^{2}.
\]

The infinite dimensional optimization problem in (\ref{eqn:2.chi2Omegadual})
thus reduces to a $(k+1)-$dimensional one, much easier to handle.

We can write the chi-square and $\chi_{n}^{2}$ through a quadratic form.

Define the vectors $\underline{\nu}_{n}$ e $\underline{\nu}$ by
\begin{align}
{\underline{\nu}_{n}}^{\prime} &  =\underline{\nu}(\mathcal{F},P_{n})^{\prime
}=\left\{  \left(  Q_{0}-P_{n}\right)  f_{1},\ldots,\left(  Q_{0}%
-P_{n}\right)  f_{k}\right\} \label{nugamma}\\
{\underline{\nu}}^{\prime} &  =\underline{\nu}(\mathcal{F},P)^{\prime
}=\left\{  \left(  Q_{0}-P\right)  f_{1},\ldots,\left(  Q_{0}-P\right)
f_{k}\right\} \nonumber\\
\underline{\gamma}_{n} &  =\underline{\gamma}_{n}(\mathcal{F})=\sqrt
{n}\left\{  \left(  P_{n}-P\right)  f_{1},\ldots,\left(  P_{n}-P\right)
f_{k}\right\}  =\sqrt{n}\left(  \underline{\nu}-\underline{\nu}_{n}\right)
.\nonumber
\end{align}
Let $S$ be the covariance matrix of $\underline{\gamma}_{n}$. Write $S_{n}$
for the empirical version of $S$, obtained substituting $P$ by $P_{n}$ in all
entries of $S$.

\begin{prop}
\label{th:matrixform} Let $\Omega$ be as in $\left(  \ref{eqn:3.omega_linear}%
\right)  $ and let $card\left\{  \mathcal{F}\right\}  $ be finite. We then have

\begin{enumerate}
\item[(i)] $\chi_{n}^{2}={\underline{\nu}}_{n}^{\prime}S_{n}^{-1}%
\underline{\nu}_{n}$

\item[(ii)] $\chi^{2}\left(  \Omega,P\right)  ={\underline{\nu}}^{\prime
}S^{-1}\underline{\nu}$
\end{enumerate}
\end{prop}

\begin{proof}

\begin{enumerate}
\item[(i)] Differentiating the function in (\ref{linear-chi_n}) with respect
to $a_{s}$, $s=0,1,\ldots,k$ yields
\begin{equation}
a_{0}=-\sum_{i=1}^{k}a_{i}P_{n}f_{i}\label{a0}%
\end{equation}
for $s=0$, while for $s>0$%
\begin{equation}
\left(  Q_{0}-P_{n}\right)  f_{s}=\frac{1}{2}\left(  a_{0}P_{n}f_{s}%
+\sum_{i=1}^{k}a_{i}P_{n}f_{i}f_{s}\right)  .\label{eqn:3.a_s}%
\end{equation}
Substituting $\left(  \ref{a0}\right)  $ in the last display,
\[
\left(  Q_{0}-P_{n}\right)  f_{s}=\frac{1}{2}\sum_{i=1}^{k}a_{i}\left(
P_{n}f_{i}f_{s}-P_{n}f_{s}P_{n}f_{i}\right)  ,
\]
i.e.
\begin{equation}
2\underline{\nu}_{n}=S_{n}\underline{a}\label{a_matrix}%
\end{equation}
where $\underline{a}^{\prime}=\left\{  a_{1},a_{2},\ldots,a_{k}\right\}  .$

Set $f_{n}^{*}=\arg\max_{< \mathcal{F}> }(Q_{0}-P_{n})f-\frac{1}{4}P_{n}f^{2}.
$ For every $h\in<\mathcal{F}> $, $\left(  Q_{0}-P_{n}\right)  h-\frac{1}%
{2}P_{n}hf_{n}^{*}=0$ . Set $h:=f_{n}^{*}$ to obtain $\left(  Q_{0}%
-P_{n}\right)  f_{n}^{*}=\frac{1}{2}P_{n}(f_{n}^{*})^{2}.$

It then follows, using $\left(  \ref{a0}\right)  $ e $\left(  \ref{a_matrix}%
\right)  ,$%
\begin{align*}
\chi_{n}^{2} &  =\left[  \left(  Q_{0}-P_{n}\right)  f_{n}^{*}-\frac{1}%
{4}P_{n}(f_{n}^{*})^{2}\right]  =\frac{1}{4}P_{n}(f_{n}^{*})^{2}=\\
&  =\frac{1}{4}P_{n}\left(  \sum_{i=1}^{k}a_{i}f_{i}-\sum_{i=1}^{k}a_{i}%
P_{n}f_{i}\right)  ^{2}=\\
&  =\frac{1}{4}\underline{a}^{\prime}S_{n}\underline{a}=\underline{\nu}%
_{n}^{\prime}S_{n}^{-1}\underline{\nu}_{n}.
\end{align*}

\item[(ii)] The proof is similar to the above one.
\end{enumerate}
\end{proof}

\subsubsection{Almost sure convergence}

Call an envelope for $\mathcal{F\ }$a function $F$ such that $\left\vert
f\right\vert \leq F$ for all $f$ in $\mathcal{F}$.

\begin{thm}
\label{th:qc2} Assume that $\chi^{2}\left(  \Omega,P\right)  $ is finite. Let
$\mathcal{F}$ be a finite class of functions as in (\ref{eqn:3.omega_linear})
with an envelope function $F$ such that $PF^{2}<\infty$.\newline Then
$\left\vert \chi_{n}^{2}-\chi^{2}\left(  \Omega,P\right)  \right\vert
\rightarrow0$, $P-a.s.$
\end{thm}

\begin{proof}
From Proposition \ref{th:matrixform},
\begin{align*}
\left\vert \chi_{n}^{2}-\chi^{2}\left(  \Omega,P\right)  \right\vert  &
=\left\vert \underline{\nu}_{n}^{\prime}S_{n}^{-1}\underline{\nu}%
_{n}-\underline{\nu}S^{-1}\underline{\nu}\right\vert \\
&  =\left\vert \underline{\nu}_{n}^{\prime}\left(  S_{n}^{-1}-S^{-1}\right)
\underline{\nu}_{n}\right\vert +\left\vert \underline{\nu}_{n}^{\prime}%
S^{-1}\underline{\nu}_{n}-\underline{\nu}^{\prime}S^{-1}\underline{\nu
}\right\vert .
\end{align*}

For $\underline{x}$ in $\mathbb{R}^{k}$ denote $\left\Vert \underline
{x}\right\Vert $ the euclidean norm. Over the space of matrices $k\times k$
introduce the algebraic norm $\left\vert \!\left\vert \!\left\vert
A\right\vert \!\right\vert \!\right\vert =\sup_{\left\Vert \underline
{x}\right\Vert \leq1}\frac{\left\Vert A\underline{x}\right\Vert }{\left\Vert
\underline{x}\right\Vert }=\sup_{\left\Vert \underline{x}\right\Vert
=1}\left\Vert A\underline{x}\right\Vert $. All entries of $A$ satisfy
$\left\vert a\left(  i,j\right)  \right\vert \leq\left\vert \!\left\vert
\!\left\vert A\right\vert \!\right\vert \!\right\vert .$ Moreover, if
$\left\vert \lambda_{1}\right\vert \leq\left\vert \lambda_{2}\right\vert
\leq\ldots\leq\left\vert \lambda_{k}\right\vert $ are the eigenvalues of $A$,
$\left\vert \!\left\vert \!\left\vert A\right\vert \!\right\vert \!\right\vert
=\left\vert \lambda_{k}\right\vert .$ Observe further that, if for all
$\left(  i,j\right)  ,$ $\left\vert a\left(  i,j\right)  \right\vert
\leq\varepsilon,$ then, for any $\underline{x}\in\mathbb{R}^{k},$ such that
$\left\Vert \underline{x}\right\Vert =1,$ $\left\Vert A\underline
{x}\right\Vert ^{2}=\sum_{i=1}^{k}\left(  \sum_{j}a\left(  i,j\right)
x_{j}\right)  ^{2}\leq\sum_{i}\sum_{j}a\left(  i,j\right)  ^{2}\left\Vert
\underline{x}\right\Vert ^{2}\leq k^{2}\varepsilon^{2}$, i.e. $\left\vert
\!\left\vert \!\left\vert A\right\vert \!\right\vert \!\right\vert \leq
k\varepsilon.$

For the first term in the RHS of the above display
\begin{align*}
A & := {\underline{\nu}_{n}}^{\prime}\left(  S_{n}^{-1}-S^{-1}\right)
\underline{\nu}_{n} ={\underline{\nu}_{n}}^{\prime}S^{-1/2}\left(
S^{1/2}S_{n}^{-1}S^{1/2}-I\right)  S^{-1/2}\underline{\nu}_{n}\\
& \leq\left\Vert {\underline{\nu}_{n}}^{\prime}S^{-1/2}\right\Vert \left\vert
\!\left\vert \!\left\vert S^{1/2}S_{n}^{-1}S^{1/2}-I\right\vert \!\right\vert
\!\right\vert \leq cost.\:k\:\left\vert \!\left\vert \!\left\vert S^{1/2}%
S_{n}^{-1}S^{1/2}-I\right\vert \!\right\vert \!\right\vert .
\end{align*}
Hence if $B:=\left\vert \!\left\vert \!\left\vert S^{1/2}S_{n}^{-1}%
S^{1/2}-I\right\vert \!\right\vert \!\right\vert $ tends to $0$ a.s., so does
$A$.

First note that
\begin{align*}
S_{n}^{-1} &  =\left(  S+S_{n}-S\right)  ^{-1}=S^{-1/2}\left(  I+S^{-1/2}%
\left(  S_{n}-S\right)  S^{-1/2}\right)  ^{-1}S^{-1/2}=\\
&  =S^{-1/2}\left[  I+\sum_{h=1}^{\infty}\left(  S^{-1/2}\left(
S-S_{n}\right)  S^{-1/2}\right)  ^{h}\right]  S^{-1/2}.
\end{align*}
Hence
\[
S^{1/2}S_{n}^{-1}S^{1/2}-I=\sum_{h=1}^{\infty}\left(  S^{-1/2}\left(
S-S_{n}\right)  S^{-1/2}\right)  ^{h},
\]
which entails
\begin{align*}
\left\vert \!\left\vert \!\left\vert S^{1/2}S_{n}^{-1}S^{1/2}-I\right\vert
\!\right\vert \!\right\vert  &  =\left\vert \!\left\vert \!\left\vert
\sum_{h=1}^{\infty}\left(  S^{-1/2}\left(  S-S_{n}\right)  S^{-1/2}\right)
^{h}\right\vert \!\right\vert \!\right\vert \leq\sum_{h=1}^{\infty}\left\vert
\!\left\vert \!\left\vert S-S_{n}\right\vert \!\right\vert \!\right\vert
^{h}\left\vert \!\left\vert \!\left\vert S^{-1/2}\right\vert \!\right\vert
\!\right\vert ^{2h}\\
&  =O_{P}\left(  \lambda_{1}^{-1}k\ \sup_{i,j}\left\vert s_{n}\left(
i,j\right)  -s\left(  i,j\right)  \right\vert \right)  ,
\end{align*}
where $\lambda_{1}$ is the smallest eigenvalue of $S.$

Since
\begin{equation}
C:=\sup_{i,j}\left|  s_{n}\left(  i,j\right)  -s\left(  i,j\right)  \right|
\leq\sup_{i,j}\left|  \left(  P_{n}-P\right)  f_{i}f_{j}\right|  +\sup
_{i}\left|  \left(  P_{n}-P\right)  f_{i}\right|  \left|  \left(
P_{n}+P\right)  F\right| \label{technical 2}%
\end{equation}
the LLN implies that $C$ tends to 0 a.s. which in turn implies that $B$ tends
to 0.

Now consider the second term. $\left|  {\underline{\nu}_{n}}^{\prime}%
S^{-1}\underline{\nu}_{n}-{\underline{\nu}}^{\prime}S^{-1}\underline{\nu
}\right|  =\left|  \left( \underline{\nu}_{n}+\underline{\nu}\right) ^{\prime
}S^{-1} \left( n^{-1/2}\underline{\gamma}_{n}\right)  \right| $ tends to 0 by LLN.
\end{proof}

\subsubsection{Asymptotic distribution of the test statistic}

Write
\[
n\chi_{n}^{2}=\sqrt{n}\underline{\nu}_{n}^{\prime}S^{-1}\sqrt{n}\underline
{\nu}_{n}+\sqrt{n}\underline{\nu}_{n}^{\prime}\left(  S_{n}^{-1}%
-S^{-1}\right)  \sqrt{n}\underline{\nu}_{n}.
\]

We then have

\begin{thm}
\label{th:dlf} Let $\Omega$ be defined by $\left(  \ref{eqn:3.omega_linear}%
\right)  $ and $\mathcal{F}$ be a finite class of linearly independent
functions with envelope function $F$ such that $PF^{2}<\infty$. Set
$k=card\{\mathcal{F}\}$. Then, under $H0$,
\[
n\chi_{n}^{2}\overset{d}{\longrightarrow}chi\left(  k\right)
\]
where $chi\left(  k\right)  $ denotes a chi-square distribution with $k$
degrees of freedom.
\end{thm}

\begin{proof}
For $P$ in $\Omega$, $\sqrt{n}\underline{\nu}_{n}=\underline{\gamma}_{n}$.
Therefore $n\chi_{n}^{2}={\underline{\gamma}_{n}}^{\prime}S^{-1}%
\underline{\gamma}_{n}+\sqrt{n}{\underline{\nu}_{n}}^{\prime}\left(
S_{n}^{-1}-S^{-1}\right)  \sqrt{n}\underline{\nu}_{n}.$

By continuity of the mapping $h\left(  \underline{y}\right)  =\underline
{y}^{\prime}S^{-1}\underline{y}$ , $\underline{\gamma}_{n}^{\prime}%
S^{-1}\underline{\gamma}_{n}$ has a limiting $chi(k)$ distribution.

It remains to prove that the second term is negligible. Indeed again from
\[
\left(  \sqrt{n}\underline{\nu}_{n}\right)  ^{\prime}\left(  S_{n}^{-1}%
-S^{-1}\right)  \left(  \sqrt{n}\underline{\nu}_{n}\right)  \leq
cst.\:k\:\left\vert \!\left\vert \!\left\vert S^{1/2}S_{n}^{-1}S^{1/2}%
-I\right\vert \!\right\vert \!\right\vert
\]
it is enough to show that $\left\vert \!\left\vert \!\left\vert S^{1/2}%
S_{n}^{-1}S^{1/2}-I\right\vert \!\right\vert \!\right\vert $ is $o_{P}\left(
1\right)  .$ This follows from (\ref{technical 2}).
\end{proof}

The asymptotic behavior of $\chi_{n}^{2}$ under $H1$ is captured by
\begin{align}
\label{consistencyH1}\sqrt{n}\left( \chi_{n}^{2}-\chi^{2}\right)  & =
-2\underline{\gamma}_{n}^{\prime}S^{-1}\underline{\nu}+\sqrt{n}\underline{\nu
}^{\prime}S^{-1/2}\left( S^{1/2}S_{n}^{-1}S^{1/2}-I\right) S^{-1/2}%
\underline{\nu}\nonumber\\
&  -2\underline{\gamma}_{n}^{\prime}S^{-1/2}\left( S^{1/2}S_{n}^{-1}%
S^{1/2}-I\right) S^{-1/2}\underline{\nu}+n^{-1/2}\underline{\gamma}%
_{n}^{\prime}S^{-1}_{n}\underline{\gamma}_{n}.
\end{align}
This proves that the test based on $n\chi_{n}^{2}$ is asymptotically consistent.

\subsection{Infinite number of linear constraints, an approach by sieves}

In various cases $\Omega$ is defined through a countable collection of linear
constraints. An example is presented in Section 3.3. Suppose thus that
$\Omega$ is defined as in $\left(  \ref{eqn:3.omega_linear}\right)  $, with
$\mathcal{F}$ an infinite class of functions
\[
\mathcal{F}=\left\{  f_{\alpha}:\mathbb{R}^{d}\rightarrow\mathbb{R}%
,\ \ \alpha\in A\right\}
\]
where \textsl{A} $\subseteq\mathbb{R}$ is a countable set of indices and
$card\left(  \mathcal{F}\right)  =card\left(  A\right)  =\infty.$ Thus
$\Omega=\left\{ Q \in M: \,Qf=Q_{0}f, \; f\in\mathcal{F}\right\} $, for some
$Q_{0} $ in $M$.

Assume that the projection $Q^{\ast}$ exists in $\Omega.$ Then, by Theorem
\ref{th:2.linear_char}
\[
f^{\ast}\in cl_{L_{1}\left(  Q^{\ast}\right)  }\left( <\mathcal{F}> \right)  .
\]

We approximate $\mathcal{F}$ through a suitable increasing sequence of classes
of functions $\mathcal{F}_{n}$ , with finite cardinality $k=k(n)$ increasing
with $n.$ Each $\mathcal{F}_{n}$ induces a subset $\Omega_{n}$ included in
$\Omega$.

Define therefore $\left\{  \mathcal{F}_{n}\right\} _{n\geq1}$\ such that
\begin{align}
\mathcal{F}_{n}  & \subseteq\mathcal{F}_{n+1}\subset\mathcal{F},\,\text{for
\ all \ \ }n\geq1\label{effe}\\
\quad\mathcal{F}  & =\bigcup\limits_{n\geq1}\mathcal{F}_{n}\label{union}%
\end{align}
and
\[
\Omega_{n}=\left\{  Q:Qf=Q_{0}f,\ \ f\in\mathcal{F}_{n}\right\}  .
\]
We thus have $\Omega_{n}\supseteq\Omega_{n+1},\,n\geq1$ and $\Omega
=\bigcap\limits_{n\geq1}\Omega_{n}.$

The idea of determining the projection of a measure $P$ on a set $\Omega$
through an approximating sequence of sets -or sieve- \ has been introduced in
this setting in ~\cite{TV93}.

\begin{thm}
[Teboulle-Vajda, 1993]With the above notation, define $Q_{n}^{\ast}$ as the
projection of $P$ on $\Omega_{n}.$ Suppose that the above assumptions on
$\left\{  \Omega_{n}\right\}  _{n\geq1}$ hold and that $\Omega_{n}%
\supseteq\Omega$\ for each $n\geq1.$ Then
\begin{equation}
\lim_{n\rightarrow\infty} \left\|  f^{\ast}-f_{n}^{\ast}\right\|
_{L_{1}\left(  P\right)  }=\lim_{n\rightarrow\infty}\left\|  \frac{dQ^{\ast}%
}{dP}-\frac{dQ_{n}^{\ast}}{dP}\right\|  _{L_{1}\left(  P\right)  }.\label{vaj}%
\end{equation}

\end{thm}

By Scheffe's Lemma this is equivalent to $\lim_{n\rightarrow\infty}%
d_{var}(Q_{n}^{\ast},Q^{\ast})=0$ where $d_{var}(Q,P):=\sup_{A\in
\mathcal{B}(\mathbb{R}^{d})}\left\vert Q(A)-P(A)\right\vert $ is the variation
distance between the p.m's $P$ and $Q.$ When $\sup_{f\in\mathcal{F}}\sup
_{x}f(x)<\infty$ then (\ref{vaj}) implies%

\begin{equation}
\lim_{k\rightarrow\infty}\chi^{2}\left(  \Omega_{n},P\right)  =\chi^{2}\left(
\Omega,P\right) \label{sieve}%
\end{equation}

The above result states that we can build a sequence of estimators of
$\chi^{2}\left(  \Omega,P\right)  $ letting $k=k(n)$ grow to infinity together
with $n.$ Define
\[
\chi_{n,k}^{2}=\sup_{f\in\mathcal{F}_{n}}\left(  Q_{0}-P_{n}\right)
f-\frac{1}{4}P_{n}f^{2}.
\]
In the following section we consider conditions on $k(n)$ entailing the
asymptotic normality of the suitably normalized sequence of estimates
$\chi_{n,k}$ when $P$ belongs to $\Omega$, i.e. under $H0$.

\subsubsection{Convergence in distribution under $H0$.}

As a consequence of Theorem $\ref{th:dlf}$, $n\chi_{n,k}^{2}$ tends to
infinity with probability 1 as $n$ $\rightarrow\infty.$

We consider the statistics
\begin{equation}
\frac{n\chi_{n,k}^{2}-k}{\sqrt{2k}}\label{chi-standard}%
\end{equation}
which will be seen to have a nondegenerate distribution as $k(n)$ tends to
infinity together with $n.$

As in ~\cite{IKL93} and ~\cite{IL90}, the main tool of the proof of the
asymptotic normality of (\ref{chi-standard}) relies on the strong
approximation of the empirical processes. We briefly recall some useful notions.

\begin{defn}
A class of functions $\mathcal{F}$ is \emph{pregaussian} if there exists a
version $B_{P}^{0}\left(  .\right)  $ of $P-$Brownian bridges uniformly
continuous in $\ell^{\infty}\left(  \mathcal{F}\right)  $, with respect to the
metric $\rho_{P}\left(  f,g\right)  =\left(  Var_{P}\left|  f-g\right|
\right)  ^{1/2}$, where $\ell^{\infty}\left(  \mathcal{F}\right)  $ is the
Banach space of all functionals $H:\mathcal{F}\rightarrow\mathbb{R}%
$\ uniformly bounded and with norm $\left\|  H\right\|  _{\mathcal{F}}%
=\sup_{f\in\mathcal{F}}\left|  H\left(  f\right)  \right|  .$
\end{defn}

For some $a>0$, let $\delta_{n}$ be a decreasing sequence with $\delta
_{n}=o\left(  n^{-a}\right)  $.

\begin{defn}
A class of functions $\mathcal{F}$ is \emph{Koml\'os-Major-Tusn\'ady}
(\emph{KMT}) with respect to $P$, with rate $\delta_{n}$\ $\left(
\mathcal{F}\in KMT\left(  \delta_{n};P\right)  \right)  $\ iff it is
pregaussian and there exists a version $B_{n}^{0}\left(  .\right)  $ of
$P-$Brownian bridges such that for any $t>0$ it holds
\begin{equation}
\Pr\left\{  \sup_{f\in\mathcal{F}}\left\vert \sqrt{n}\left(  P_{n}-P\right)
f-B_{n}^{0}\left(  f\right)  \right\vert \geq\delta_{n}\left(  t+b\log
n\right)  \right\}  \leq ce^{-\theta t},\label{KMT}%
\end{equation}
where the positive constants $b$, $c$ and $\theta$ depend on $\mathcal{F}$ only.
\end{defn}

We refer to \cite{BOR81}, \cite{Mas89}, \cite{BM89}, and \cite{Kol94} for
examples of classical and useful classes of KMT classes, together with
calculations of rates; we will use the fact that a KMT class is also a Donsker class.

From $\left(  \ref{KMT}\right)  $ and Borel-Cantelli lemma it follows that
\begin{equation}
\sup_{f \in\mathcal{F}}\left|  \gamma_{n}(f)-B_{n}^{0}(f)\right| =O\left(
\delta_{n}\log n\right) \label{KMT2}%
\end{equation}
a.s. where, with the same notation as in the finite case (see (\ref{nugamma}%
)), $\gamma_{n}(f)=\sqrt{n}(P_{n}-P)f$ is the empirical process indexed by $f
\in\mathcal{F}$.

Let $\left\{  \mathcal{F}_{n}\right\} _{n\geq1}$ be a sequence of classes of
linearly independent functions satisfying $\left(  \ref{effe}\right)  $.

For any $n$, set $\underline{\gamma}_{n,k}=\underline{\gamma}_{n}%
(\mathcal{F}_{n})$ (resp. $\underline{B}_{n,k}^{0}$) the $k-$dimensional
vector resulting from the projection of the empirical process $\gamma_{n}$
(resp. of the $P-$Brownian bridge $B_{n}^{0}$) defined on $\mathcal{F}$ to the
subset $\mathcal{F}_{n}$. Then, if $\mathcal{F}_{n}=\left\{ f_{1}^{(n)}%
,\ldots,f_{k}^{(n)}\right\} $, $\underline{\gamma}_{n,k}=\left\{ \gamma
_{n}(f_{1}^{(n)}),\ldots,\gamma_{n}(f_{k}^{(n)})\right\} $ and $\underline
{B}^{0}_{n,k}=\left\{ B_{n}^{0}\left( f_{1}^{(n)}\right) ,\ldots,B_{n}%
^{0}\left( f_{k}^{(n)}\right) \right\} $. Denote $S_{k}$ the covariance matrix
of the vector $\underline{\gamma}_{n,k}$ and $S_{n,k}$ its empirical
covariance matrix. Let $\lambda_{1,k}$ be the smallest eigenvalue of $S_{k}$.

\begin{thm}
\label{th:dli}Let $\mathcal{F}$ have an envelope $F$ and be $KMT\left(
\delta_{n};P\right)  $\ for some sequence $\delta_{n}\downarrow0.$ Define
further a sequence $\left\{  \mathcal{F}_{n}\right\}  _{n\geq1}$ of classes of
linearly independent functions satisfying $\left(  \ref{effe}\right)  $.

Moreover, let $k$ satisfy
\begin{align}
\lim_{n\rightarrow\infty}k(n) &  =\infty\nonumber\\
\lim_{n\rightarrow\infty}\lambda_{1,k}^{-1/2}\,k^{1/2}\delta_{n}\log n &
=0\label{B}\\
\lim_{n\rightarrow\infty}\lambda_{1,k}^{-1}\,k^{3/2}n^{-1/2} &  =0.\label{D}%
\end{align}
Then under $H0$
\[
\frac{n\chi_{n,k}^{2}-k}{\sqrt{2k}}\overset{d}{\longrightarrow}N\left(
0,1\right)  .
\]

\end{thm}

\begin{proof}
By Proposition \ref{th:matrixform}
\begin{align*}
\frac{n\chi^{2}_{n,k}-k}{\sqrt{2k}} &  =\frac{\left(  \underline{B}_{n,k}%
^{0}\right)  ^{\prime}S_{k}^{-1}\underline{B}_{n,k}^{0}-k}{\sqrt{2k}}+2\left(
2k\right)  ^{-1/2}\left(  \underline{B}_{n,k}^{0}\right)  ^{\prime}S_{k}%
^{-1}\left(  \underline{\gamma}_{n,k}-\underline{B}_{n,k}^{0}\right) \\
&  +\left(  2k\right)  ^{-1/2}\left(  \underline{\gamma}_{n,k}-\underline
{B}_{n,k}^{0}\right)  ^{\prime}S_{k}^{-1}\left(  \underline{\gamma}%
_{n,k}-\underline{B}_{n,k}^{0}\right) \\
&  +\left(  2k\right)  ^{-1/2}{\underline{\gamma}_{n,k}}^{\prime}\left(
S_{n,k}^{-1}-S_{k}^{-1}\right)  \underline{\gamma}_{n,k}\\
&  =A+B+C+D.
\end{align*}
The first term above can be written
\[
A=\frac{\left(  \underline{B}_{n,k}^{0}\right)  ^{\prime}S_{k}^{-1}%
\underline{B}_{n,k}^{0}-k}{\sqrt{2k}}=\frac{\sum\limits_{i=1}^{k}\left(
Z_{i}^{2}-EZ_{i}^{2}\right)  }{\sqrt{kVarZ_{i}^{2}}}%
\]
which converges weakly to the standard normal distribution by the CLT applied
to the i.i.d.standard normal r.v's $Z_{i}.$

As to the term $C$ it is straightforward that $C=o(B)$. From the proof of
Theorem \ref{th:dlf}, D goes to zero if $\lambda_{1,k}^{-1}k^{1/2} \left(
\sup_{i,j}\left\vert s_{n,k}\left(  i,j\right)  -s_{k}\left(  i,j\right)
\right\vert \right)  \left\Vert \underline{\gamma}_{n,k}^{\prime}S_{k}%
^{-1/2}\right\Vert ^{2}=o_{P}(1).$ Since, using (\ref{D}) and (\ref{effe}),
$\sup_{i,j}\left\vert s_{n,k}\left(  i,j\right)  -s_{k}\left(  i,j\right)
\right\vert \leq\sup_{f,g\in\mathcal{F}}\left\vert \left(  P_{n}-P\right)
fg\right\vert $

$+\sup_{f\in\mathcal{F}}\left\vert \left(  P_{n}-P\right)  f\right\vert
\left\vert \left(  P_{n}+P\right)  F\right\vert $ $=O_{P}\left(
n^{-1/2}\right)  $, and considering that $\left\Vert \underline{\gamma}%
_{n,k}^{\prime}S_{k}^{-1/2}\right\Vert ^{2}=O_{P}(k)$, we are done.

For B, $\left\vert \left(  \underline{B}_{n,k}^{0}\right)  ^{\prime}S_{k}%
^{-1}\left(  \underline{\gamma}_{n,k}-\underline{B}_{n,k}^{0}\right)
\right\vert =\left\vert \left(  \underline{B}_{n,k}^{0}\right)  ^{\prime}%
S_{k}^{-1/2}S_{k}^{-1/2}\left(  \underline{\gamma}_{n,k}-\underline{B}%
_{n,k}^{0}\right)  \right\vert $

$\leq\left\Vert \left(  \underline{B}_{n,k}^{0}\right)  ^{\prime}S_{k}%
^{-1/2}\right\Vert \left\Vert S_{k}^{-1/2}\left(  \underline{\gamma}%
_{n,k}-\underline{B}_{n,k}^{0}\right)  \right\Vert $ $=\sqrt{\sum_{i=1}%
^{k}Z_{i}^{2}}\left\Vert S_{k}^{-1/2}\left(  \underline{\gamma}_{n,k}%
-\underline{B}_{n,k}^{0}\right)  \right\Vert $ where, as used in A, $Z_{i}%
^{2}$\ are i.i.d. with a $\chi^{2}$ distribution with 1 df. Hence $\sqrt
{\sum_{i=1}^{k}Z_{i}^{2}}=O_{P}\left(  k^{1/2}\right)  .$ Further
\begin{align*}
\left\Vert S_{k}^{-1/2}\left(  \underline{\gamma}_{n,k}-\underline{B}%
_{n,k}^{0}\right)  \right\Vert  & \leq\left| \!\left| \!\left|  S_{k}%
^{-1/2}\right| \!\right| \!\right|  \cdot\left\Vert \underline{\gamma}%
_{n,k}-\underline{B}_{n,k}^{0} \right\Vert \\
& \leq\lambda_{1,k}^{-1/2}\,k^{1/2}\sqrt{\frac{1}{k}\sum_{i=1}^{k}\left(
\underline{\gamma}_{n,k}(f_{i})-\underline{B}_{n,k}^{0}(f_{i})\right)  ^{2}}\\
& \leq\lambda_{1,k}^{-1/2}\,k^{1/2}\sup_{f \in\mathcal{F}}\left|  \gamma
_{n}(f)-B_{n}^{0}(f)\right|
\end{align*}
from which $B=O_{P}\left(  \lambda_{1,k}^{-1/2}\,k^{1/2}\delta_{n}\log
n\right)  =o_{P}\left(  1\right)  $ if (\ref{B}) holds. We have used the fact
that $P$ belongs to $\Omega$ in the last evaluation of $B$.
\end{proof}

\begin{rem}
\label{rem:consistency} Under $H1$, using the relation $\underline{\nu}%
_{n}=\underline{\nu}-n^{-1/2}\underline{\gamma}_{n,k}$, we can write
\begin{align*}
\frac{n\chi^{2}_{n,k}-k}{\sqrt{2k}} & =(2k)^{-1/2}\left( \underline{\gamma
}_{n,k}^{\prime}S_{n,k}^{-1}\underline{\gamma}_{n,k}-k\right) +(2k)^{-1/2}%
\left( n\underline{\nu}^{\prime}S_{n,k}^{-1}\underline{\nu}-2\sqrt
{n}\underline{\gamma}_{n,k}^{\prime}S_{n,k}^{-1}\underline{\nu}\right) \\
& =(2k)^{-1/2}\left( n\underline{\nu}^{\prime}S_{n,k}^{-1}\underline{\nu
}-2\sqrt{n}\underline{\gamma}_{n,k}^{\prime}S_{n,k}^{-1}\underline{\nu}\right)
+O_{P}(1)
\end{align*}
where the $O_{P}(1)$ term captures $(2k)^{-1/2}\left( \underline{\gamma}%
_{n,k}^{\prime}S_{n,k}^{-1}\underline{\gamma}_{n,k}-k\right) $ that coincides
with the test statistic $\frac{n\chi^{2}_{n,k}-k}{\sqrt{2k}}$ under $H0$. We
can bound the first term from below by
\begin{align*}
& (2k)^{-1/2}n\underline{\nu}^{\prime}S_{k}^{-1}\underline{\nu}\left(
1-O_{P}\left( \left| \!\left| \!\left|  S_{k}^{1/2}S_{n,k}^{-1}S_{k}%
^{1/2}-I\right| \!\right| \!\right| \right) \right) \\
& \quad- (2n/k)^{1/2}\left\| \underline{\gamma}_{n,k}^{\prime}S_{k}%
^{-1/2}\right\| \left\| \underline{\nu}^{\prime}S_{k}^{-1/2}\right\|  \left(
1+O_{P}\left( \left| \!\left| \!\left|  S_{k}^{1/2}S_{n,k}^{-1}S_{k}%
^{1/2}-I\right| \!\right| \!\right| \right) \right) \\
& = O_{P}\left(  nk^{-1/2} \right) -O_{P}(n^{1/2}).
\end{align*}
Hence, if (\ref{B}) and (\ref{D}) are satisfied then the test statistic is
asymptotically consistent also for the case of an infinite number of linear constraints.
\end{rem}

In both conditions (\ref{B}) and (\ref{D}) the value of $\lambda_{1,k}$
appears, which cannot be estimated without any further hypothesis on the
structure of the class $\mathcal{F}$. However, for concrete problems, once
defined $\mathcal{F}$ it is possible to give bounds for $\lambda_{1}$,
depending on $k$. This is what will be shown in the last section, for a
particular class of goodness of fit tests.

\subsection{Application: testing marginal distributions}

\label{sec:example}

Let $P$ be an unknown distribution on $\mathbb{R}^{d}$ with density bounded by
below. We consider goodness-of-fit tests for the marginal distributions
$P_{1},...,P_{d}$ of $P$ on the basis of an i.i.d. sample $(X_{1},...,X_{n}).
$

Let thus $Q_{1}^{0},\ldots,Q_{d}^{0}$ denote $d$ distributions on $\mathbb{R}%
$. The null hypothesis writes $H0:$ $P_{j}=Q_{j}^{0}$ for $j=1,...,d.$ That is
to say that we simultaneously test goodness-of-fit of the marginal laws
$P_{1},\ldots,P_{d}$ to the laws $Q_{1}^{0},\ldots,Q_{d}^{0}.$ Through the
transform $P^{\prime}(y_{1},\ldots,y_{d})=P\left(  \left(  Q_{1}^{0}\right)
^{-1}\left(  y_{1}\right)  ,\ldots,\left(  Q_{d}^{0}\right)  ^{-1}\left(
y_{d}\right)  \right)  $ we can restrict the analysis to the case when all
p.m's have support $[0,1]^{d}$ and marginal laws uniform in $[0,1]$ under H0.
So without loss of generality we write $Q_{0}$ for the uniform distribution on
$[0,1]$.

P.J. Bickel, Y. Ritov and J.A. Wellner \cite{BRW91} focused on the estimation
of linear functionals of the probability measure subject to the knowledge of
the marginal laws in the case of r.v.'s with a.c. distribution, letting the
number of cells grow to infinity.

Define\ the class
\[
\mathcal{F}:=\left\{  {\LARGE 1}_{u,j}:\left[  0,1\right]  ^{d}\longrightarrow
\left\{  0,1\right\}  ,\text{ }j=1,\ldots,d\text{, \ }u\in\lbrack0,1]\right\}
\]
where ${\LARGE 1}_{u,j}\left(  x_{1},\ldots,x_{d}\right)  =\left\{
\begin{array}
[c]{c}%
1\ ,\quad x_{j}\leq u\\
0\ ,\quad x_{j}>u
\end{array}
\right.  $.

Let $\Omega$ be the set of all p.m's on $\left[  0,1\right]  ^{d}$ with
uniform marginals, i.e
\begin{equation}
\Omega=\left\{  Q\in M_{1}\left(  \left[  0,1\right]  ^{d}\right)  \text{
\ \ such that \ }Qf=\int_{\left[  0,1\right]  ^{d}}f(x)dx\,,\ \ f\in
\mathcal{F}\right\}  .\label{Omega marge}%
\end{equation}

This set $\Omega$ has the form of $\left(  \ref{eqn:3.omega_linear}\right)  ,$
where $\mathcal{F}$ is the class of characteristic functions of intervals,
which is a KMT class with rate $\delta_{n}=n^{-1/2d}$ ($\delta_{n}=\sqrt{n} $
if $d=2$) ; see \cite{BM89}.

We now build the family $\mathcal{F}_{n}$ satisfying (\ref{effe}) and
(\ref{union}).

Let $m=m(n)$ tend to $+\infty$ with $n$. Let $0< u_{1}< \ldots< u_{m}< 1$ and
$\left\{ \mathcal{U}^{(n)}\right\} $ be the $m\cdot d$ points in $[0,1]^{d}$
with coordinates in $\{u_{1},\ldots,u_{m}\}$.

Let $\mathcal{F}_{n}$ denote the class of characteristic functions of the
$d-$dimensional rectangles $[\underline{0},\underline{u}]$ for $\underline
{u}\in\mathcal{U}^{(n)}$. Hence $card\{\mathcal{F}_{n}\}=k=m\cdot d$.

Namely,
\begin{equation}
\mathcal{F}_{n}=\left\{  {\LARGE 1}_{u_{i},j}:\left[  0,1\right]
^{d}\rightarrow\left\{  0,1\right\}  ,\text{ }j=1,\ldots,d\text{, \ }u_{i}%
\in\left(  0,1\right)  ,\,u_{i}<u_{i+1}\,,\ i=1,\ldots,m\right\}
,\label{effe_k}%
\end{equation}
which satisfies $\mathcal{F}_{n}\subseteq\mathcal{F}_{n+1}$ for all $n\geq1 $
(i.e. (\ref{effe})) and $\mathcal{F}=\bigcup_{n\geq1}\mathcal{F}_{n}$ (i.e.
(\ref{union}))$.$

The sequence $\left\{  \mathcal{F}_{n}\right\}  _{n\geq1}$ and the class
$\mathcal{F}$ satisfy conditions of Theorem $\ref{th:dli}$: $\mathcal{F}$ (and
consequently each $\mathcal{F}_{n}$) has envelope function $F=1$ and
$RF^{h}=1,$ for all $R$ in $\ M_{1}\left(  \left[  0,1\right]  ^{d}\right)  $
and $h$ in $\mathbb{R}$.

In order to establish a lower bound for $\lambda_{1,k},$ the smallest
eigenvalue of $S_{k},$ we will impose that the volumes of the cells in the
grid defined by the $u_{i}^{(n)}$ do not shrink too rapidly to 0. Suppose that
the intervals $(u_{i},u_{i+1}]$ are such that
\begin{equation}
0<\lim_{n\rightarrow\infty}\inf\min_{i=1,...,m-1}k\left(  u_{i+1}%
-u_{i}\right)  \leq\lim_{n\rightarrow\infty}\sup\max_{i=1,...,m-1}k\left(
u_{i+1}-u_{i}\right)  <\infty.\label{grid}%
\end{equation}

\begin{rem}
Condition for the sequence $\mathcal{F}_{n}$ to converge to $\mathcal{F}$
coincides with (F2) and (F3) in ~\cite{BRW91}.
\end{rem}

We first obtain an estimate for the eigenvalue $\lambda_{1,k}.$ The final
result of this step is stated in Lemma \ref{lm:eigen2} below.

Let $P$ belong to $\Omega.$ Let us then write the matrix $S_{k}$. We have
$P{\LARGE 1}_{u_{i},j}=Q_{0}{\LARGE 1}{}_{u_{i},j}=u_{i}$ for $i=1,...,m$ and
$j=1,...,d$ . Set $P{\LARGE 1}_{u_{i},j}{\LARGE 1}_{u_{l},h}=P\left(
X_{j}\leq u_{i},\,X_{h}\leq u_{l}\right)  $, for every $h,j=1,...,d$ and
$l,i=1,\ldots,m$. When $j=h$ then $P{\LARGE 1}_{u_{i},j}{\LARGE 1}_{u_{l},j}
=P\left(  X_{j}\leq u_{i}\wedge u_{l}\right)  =u_{i}\wedge u_{l}.$

Consider for the vector of functions $f_{j}$ the following ordering
\[
\left(  f_{1},\ldots,f_{m},f_{m+1},\ldots,f_{2m},\ldots,f_{\left(  d-1\right)
m+1},\ldots,f_{dm},\right)  =\left(  {\LARGE 1}_{u_{1},1}\ldots,{\LARGE 1}%
_{u_{m},1},{\LARGE 1}_{u_{1},2},\ldots,{\LARGE 1}_{u_{m},d}\right)  .
\]
The generic term of $S_{k}$ writes
\begin{align*}
s_{k}\left(  u,v\right)   &  =s_{k}\left(  (j-1)m+i,(h-1)m+l\right)
=P{\LARGE 1}_{u_{i},j}{\LARGE 1}_{u_{l},h}-P{\LARGE 1}_{u_{i},j}%
P{\LARGE 1}_{u_{l},h}\\
&  =\left\{
\begin{array}
[c]{ll}%
u_{i}-u_{i}^{2}\text{ \ \ }, & \text{ \ if \ }j=h,\,i=l\\
P(X_{j}\leq u_{i}\wedge u_{l})-u_{i}u_{l}\text{ \ \ }, & \text{ \ if
\ }j=h,\,i\neq l\\
P(X_{j}\leq u_{i},X_{h}\leq u_{l})-u_{i}u_{l}\left(  i,l\right)  -p_{i}p_{l} &
\text{ \ if \ }j\neq h
\end{array}
\right.
\end{align*}

We make use of the class of functions
\begin{align*}
\mathcal{F}_{n}^{\delta}  & =\left\{  f_{j\cdot i}-f_{j(i-1)},\;i=1,\ldots
,m,\;j=1\ldots,d,\;\;f_{h}\in\mathcal{F}_{n},\;f_{0}=0\right\} \\
& =\left\{  {\LARGE 1}_{A_{i}^{j}},\;i=1\ldots,m,\;j=1,\ldots,d\right\}  .
\end{align*}

In the above display the $\{A_{i}^{j}\}_{i=1,\ldots,m}$ describe the partition
of $[0,1],$ the support of the marginal distribution $P_{j}$, induced by the
vector $\left\{  u_{i}\right\}  _{i\leq m}$. Namely we have for every
$j=1,\ldots,d$, $A_{i}^{j}\cap A_{l}^{j}=\emptyset$, $i\not =l$, $\cup
_{i=1}^{m+1}A_{i}^{j}=[0,1]$, with $A_{m+1}^{j}=[u_{m},1],$ for all $j$.

Set $S_{k}^{\delta}$ the covariance matrix of the vector $\underline{\gamma
}_{n}^{\delta}=\underline{\gamma}_{n}(\mathcal{F}^{\delta})$ and consider the
vectors $\underline{\nu}^{\delta}$ and $\underline{\nu}_{n}^{\delta}$ defined
as in (\ref{nugamma}).

$S_{k}^{\delta}$ has $((j-1)m+i,(h-1)m+l)-$th component equal to $P_{A_{i}%
^{j}}{}_{A_{l}^{h}}-P_{A_{i}^{j}}P_{A_{l}^{h}}$ , which is%
\[
\left\{
\begin{array}
[c]{ll}%
p_{i}-p_{i}^{2}, & \text{ \ if \ }j=h,\,i=l\\
-p_{i}p_{l}, & \text{ \ if \ }j=h,\,i\neq l\\
P(u_{i-1}\leq X_{j}<u_{i},u_{l-1}\leq X_{h}<u_{l})-p_{i}p_{l}, & \text{ \ if
\ }j\neq h
\end{array}
\right.
\]
where we have written $p_{i}=P(u_{i-1}\leq X_{j}<u_{i})=u_{i}-u_{i-1}$, for
all $j=1,\ldots,d$.

$\chi^{2}$ (and $\chi_{n}^{2}$) can be written using $\mathcal{F}_{n}^{\delta
}$ instead of $\mathcal{F}_{n}$:
\[
\chi^{2}(\Omega,P)=\underline{\nu}^{\prime}S_{k}^{-1}\underline{\nu}=\left(
\underline{\nu}^{\delta}\right)  ^{\prime}\left(  S_{k}^{\delta}\right)
^{-1}\left(  \underline{\nu}^{\delta}\right)  .
\]

Let $M$ be the diagonal $d-$block matrix with all diagonal blocks equal to the
unit inferior triangular $(m\times m)$ matrix. Then $\underline{\nu
}=M\underline{\nu}^{\delta}$.

On the other hand, after some algebra it can be checked that $S_{k}=M
S_{k}^{\delta}M^{\prime}$.

Thus ${\underline{\nu}^{\delta}}^{\prime}{(S_{k}^{\delta})}^{-1}\underline
{\nu}^{\delta}={\underline{\nu}}^{\prime}(M^{\prime})^{-1}M^{\prime}S_{k}%
^{-1}M(M)^{-1}\underline{\nu}=\chi^{2}$. Similar arguments yield $\chi_{n}%
^{2}={\underline{\nu}_{n}^{\delta}}^{\prime}(S_{n,k}^{\delta})^{-1}%
\underline{\nu}_{n}^{\delta}$.

The matrix $M$ has all eigenvalues equal to one. This allows us to write, for
$\lambda_{1,\delta}$ the minimum eigenvalue of $S_{k}^{\delta}$:
\[
\lambda_{1,k}\leq\min_{x}\frac{x^{\prime}S_{k}x}{\Vert x\Vert^{2}}\leq\min
_{y}\frac{y^{\prime}S_{k}^{\delta}y}{\Vert y\Vert^{2}}\max_{x}\frac{\Vert
Mx\Vert^{2}}{\Vert x\Vert^{2}}=\lambda_{1,\delta}\leq\min_{y}\frac{y^{\prime
}S_{k}y}{\Vert y\Vert^{2}}\max_{x}\frac{\Vert M^{-1}x\Vert^{2}}{\Vert
x\Vert^{2}}=\lambda_{1,k}.
\]

We will now consider the covariance matrix of $\underline{\gamma}_{n}^{\delta
}$ under H0, when the underlying distribution is $Q_{0}$, i.e. the uniform
distribution on $[0,1]^{d}$. Denote this matrix $S_{k}^{0}$. We then have

\begin{lem}
\label{th:4.eigen1} If $P\in\Omega$, then

(i) $S_{k}^{0}=D^{1/2}(I-V)D^{1/2},$ where $D$ and $V$ are both diagonal block
matrices with diagonal blocks equal to $diag\left\{ p_{i}\right\}
_{i=1,\ldots,m} $ and to $U=\left\{ \sqrt{p_{i} p_{l}}\right\} _{i=1,\ldots
,m,\, l=1,\ldots,m}$ respectively.

(ii) The $(m\times m)$ matrix $U$ has eigenvalues equal to
\[
\lambda_{U} = \left\{
\begin{array}
[c]{ll}%
(1-p_{m+1})=\sum_{i=1}^{m} p_{i} & \mbox{ with cardinality }1\\
0 & \mbox{ with cardinality }m-1.
\end{array}
\right.
\]
Moreover $(I-U)^{-1}=(I+\frac{1}{p_{m+1}}U)$.

(iii) For any eigenvalue $\lambda$ of $S_{k}^{0}$ it holds
\[
p_{m+1}\min_{1\leq i\leq m}p_{i}\leq\lambda\leq\max_{1\leq i\leq m}p_{i}.
\]

\end{lem}

\begin{proof}
(i) This can easily be checked through some calculation.

(ii) First notice that
\begin{equation}
\label{eq:4.U^2}U^{2}=(1-p_{m+1})U
\end{equation}
Formula (\ref{eq:4.U^2}) implies that at least one eigenvalue equals
$(1-p_{m+1})$. On the other hand, summing up all diagonal entries in $U$ we
get $trace(U)=\sum_{i=1}^{m}p_{i}=1-p_{m+1}$. This allows us to conclude that
there can be only one eigenvalue equal to $1-p_{m+1}$ while the other must be zero.

For the second statement, by Taylor expansion of $(1-x)^{-1}$, $(I-U)^{-1}%
=I+\sum_{h=1}^{\infty}U^{h}$.

Then, using recursively (\ref{eq:4.U^2}), $(I-U)^{-1}=I+U\sum_{h=1}^{\infty
}(1-p_{m+1})^{h}=U+\frac{1}{p_{m+1}}U$.

(iii) For any eigenvalue $\lambda$ of $S_{k}^{0}$ we have:
\[
\lambda\leq\lambda_{k,k} =\left| \!\left| \!\left|  S_{k}^{0} \right|
\!\right| \!\right|  \leq\left| \!\left| \!\left|  D^{1/2}\right| \!\right|
\!\right| ^{2} \left| \!\left| \!\left|  (I-V) \right| \!\right| \!\right|
=\max_{1\leq i \leq m}p_{i} \left( 1-\inf_{\|\underline{x}\|=1}\underline
{x}^{\prime}V\underline{x}\right) =\max_{1\leq i \leq m}p_{i}%
\]
where for the last identity we have used the fact that the eigenvalues of $V$
coincide with the eigenvalues of $U$ with order multiplied by $d$.

For the opposite inequality consider
\begin{align*}
\lambda^{-1}  & \leq\lambda_{1,k}^{-1}=\left\vert \!\left\vert \!\left\vert
{S_{k}^{0}}^{-1}\right\vert \!\right\vert \!\right\vert \leq\left\vert
\!\left\vert \!\left\vert D^{-1}\right\vert \!\right\vert \!\right\vert
\left\vert \!\left\vert \!\left\vert \left(  I+\frac{1}{p_{m+1}}V\right)
\right\vert \!\right\vert \!\right\vert \\
& =\left(  \max_{1\leq i\leq m}p_{i}^{-1}\right)  \left(  1+\frac{1}{p_{m+1}%
}(1-p_{m+1})\right)  =\left(  \min_{1\leq i\leq m}p_{i}\right)  ^{-1}%
p_{m+1}^{-1}.
\end{align*}

\end{proof}

\begin{lem}
\label{lm:eigen2} Suppose that $P$ has density on $[0,1]^{d}$ bounded from
below by $\alpha>0.$ Then the smallest eigenvalue of $S_{k}^{\delta}$, and
consequently $\lambda_{1,k}$, is bounded below by $p_{m+1}\,\min_{1\leq i\leq
m}p_{i}.$
\end{lem}

\begin{proof}
Write $s_{k}^{\delta}\left(  u,v\right)  $ for the $\left(  u,v\right)  -$th
element of $S_{k}^{\delta}$. We have, for $P\in\Omega$, i.e. if $Pf=Q_{0}f $,
for every $f\in\mathcal{F}_{n}^{\delta}$:
\begin{align*}
s_{k}^{\delta}\left(  (j-1)m+i,(h-1)m+l\right)   &  =s_{k}^{\delta}\left(
u,v\right)  =Pf_{u}f_{v}-Pf_{u}Pf_{v}=\\
&  =P\left(  f_{u}-Q_{0}f_{u}\right)  \left(  f_{v}-Q_{0}f_{v}\right)
=P\left(  \overline{f_{u}}\overline{f_{v}}\right)
\end{align*}
where $\overline{f_{u}}=f_{u}-Q_{0}f_{u}.$ For each vector $\underline{a}%
\in\mathbb{R}^{d\cdot m}$ it holds then%

\begin{align*}
\underline{a}^{\prime}S_{k}^{\delta} \underline{a} &  =\sum_{u=1}^{dm}%
\sum_{v=1}^{dm}a_{u}a_{v}P\left(  \overline{f_{u}}\overline{f_{v}}\right)
=P\left(  \sum_{u=1}^{dm}a_{u}\overline{f_{u}}\right)  ^{2}\\
&  =\int_{\left[  0,1\right]  ^{d}}\left(  \sum_{u=1}^{dm}a_{u}\overline
{f_{u}}\right)  ^{2}dP\geq\alpha\int_{\left[  0,1\right]  ^{d}}\left(
\sum_{u=1}^{dm}a_{u}\overline{f_{u}}\right)  ^{2}dQ_{0}=\\
&  =\alpha\,\underline{a}^{\prime}\left\{  Q_{0}\left(  f_{u}-Q_{0}%
f_{u}\right)  \left(  f_{v}-Q_{0}f_{v}\right)  \right\}  _{u,v}\underline
{a}=\alpha\,\underline{a}^{\prime}S_{k}^{0} \underline{a}%
\end{align*}
On the other hand the preceding inequality implies
\begin{equation}
\inf_{\underline{a}}\frac{\underline{a}^{\prime}S_{k}^{\delta} \underline{a}%
}{\left\|  \underline{a}\right\|  ^{2}}\geq\alpha\,\inf_{\underline{a}}%
\frac{\underline{a}^{\prime}S_{k}^{0} \underline{a}}{\left\|  \underline
{a}\right\|  ^{2}}\label{sigma1}%
\end{equation}
that is a lower bound for the smallest eigenvalue of $S_{k}^{\delta}$
depending on the smallest eigenvalue of $S_{k}^{0}$.

Apply Lemma \ref{th:4.eigen1} (iii) to get the lower bound for $\lambda_{1}$.
\end{proof}

\begin{rem}
\label{rem:alpha} Existence of $\alpha>0$ such that the density of $P$ in
$[0,1]^{d}$ is bounded below by $\alpha$ seems necessary for this kind of
approach; see assumption (P3) in \cite{BRW91}.
\end{rem}

From Theorem \ref{th:dli} and using (\ref{grid}) in order to evaluate
$p_{m+1}\,\min_{1\leq i\leq m}p_{i}$, together with the fact that the class
$\mathcal{F}$ is KMT\ with rate $\delta_{n}=n^{-1/2}$ we obtain

\begin{thm}
\label{sec4:thm F infinite}Let (\ref{grid}) hold. Assume that $P$ belongs to
$\Omega$ defined by (\ref{Omega marge}) and has a density bounded below by
some positive number. Let further $k=d\cdot m(n)$ be a sequence such that
$\lim_{n\rightarrow\infty}k=\infty$ and $\lim_{n\rightarrow\infty}
k^{7/2}n^{-1/2}=0$

Then $\frac{n\chi_{n,k}^{2}-k}{\sqrt{2k}}=\frac{n\underline{\gamma}%
_{n,k}^{\prime}S_{n,k}^{-1}\underline{\gamma}_{n,k}-k}{\sqrt{2k}}$ has
limiting normal standard distribution$.$
\end{thm}

In the last part of this Section we intend to show that conditions in Theorem
\ref{sec4:thm F infinite} can be weakened for small values of $d.$ When $d=2$
the rate for $k=2m$ is achieved when condition (\ref{B}) holds.

We consider the case when $d=2$; for larger values of $d$, see Remark
\ref{rmk:d>2}.

In order to make the notation more clear, define $p_{i,j}$ and $N_{i,j},$
respectively, $P(A_{i}^{1}\times A_{j}^{2})$ and $nP_{n}(A_{i}^{1}\times
A_{j}^{2})$, where the events $A_{i}^{h}$, $h=1,2$, $i=1,\ldots,m$ are as
above. The marginal distributions will be denoted $\ p_{i,\cdot}=p_{\cdot
,i}=p_{i}$ (since H0 holds), and the empirical marginal distributions by
$N_{i,\cdot}/n$ and $N_{{\cdot,i}}/n$.

Turning back to the proof of Theorem \ref{th:dli} we see that condition
(\ref{D}) is used in order to ensure that $\underline{\gamma}_{n,k}^{\prime
}(S_{n,k}^{-1}-S_{k}^{-1})\underline{\gamma}_{n,k}$ goes to $0$ in probability
as $n$ tends to infinity, while condition (\ref{B}) implies the convergence of
$\frac{\underline{\gamma}_{n,k}^{\prime}S_{k}^{-1}\underline{\gamma}_{n,k}%
-2m}{\sqrt{4m}}$ to the standard normal distribution.


Let
\begin{align*}
\mathcal{Q}=\Biggl\{Q\in M_{1}([0,1]^{2})\Biggr.  & :\left.  \sum_{i=1}%
^{m+1}q_{i,j}=p_{\cdot,j}=q_{\cdot,j}^{0}=u_{j+1}-u_{j},\:j=1,\ldots
,m+1;\right. \\
& \Biggl.\sum_{j=1}^{m+1}q_{i,j}=p_{i,\cdot}=q_{i,\cdot}^{0}=u_{i+1}%
-u_{i},\:i=1,\ldots,m+1\Biggr\},
\end{align*}
where $q_{i,j}^{0}=Q^{0}(A_{i}^{1}\times A_{j}^{2})=(u_{i+1}-u_{i}%
)(u_{j+1}-u_{j})$.

\begin{lem}
When $P\in\Omega$, it holds
\begin{align}
n\chi_{n,k}^{2}  & =\min_{Q\in\mathcal{Q}}\sum_{i=1}^{m+1}\sum_{j=1}%
^{m+1}\frac{(nq_{i,j}-N_{i,j})^{2}}{N_{i,j}}\mathbb{I}_{N_{i,j}>0}%
\label{chi_n_BRW}\\
\underline{\gamma}_{n,k}^{\prime}S_{k}^{-1}\underline{\gamma}_{n,k}  &
=\min_{Q\in\mathcal{Q}}\sum_{i=1}^{m+1}\sum_{j=1}^{m+1}\frac{(nq_{i,j}%
-N_{i,j})^{2}}{np_{i,j}}\label{chi_n_bar}%
\end{align}

\end{lem}

\begin{proof}
We prove (\ref{chi_n_BRW}), since the proof of (\ref{chi_n_bar}) is similar.
Following \cite{BRW91} the RHS in (\ref{chi_n_BRW}) is
\[
\sum_{i=1}^{m+1}\sum_{j=1}^{m+1}N_{i,j}(a_{i}+b_{j})^{2},
\]
where the vectors $a$ and $b$ $\in\mathbb{R}^{m+1}$ are solutions of the
equations
\begin{align*}
a_{i}\frac{N_{i,\cdot}}{n} &  =p_{i}-\frac{N_{i,\cdot}}{n}-\sum_{j=1}%
^{m+1}b_{j}\frac{N_{i,j}}{n},\;\;\;i=1,\ldots,m+1,\\
b_{j}\frac{N_{\cdot,j}}{n} &  =p_{j}-\frac{N_{\cdot,j}}{n}-\sum_{i=1}%
^{m+1}a_{i}\frac{N_{i,j}}{n},\;\;\;j=1,\ldots,m+1.
\end{align*}
Let $\underline{a}=(\tilde{a}_{1},\ldots,\tilde{a}_{m},\tilde{b}_{1}%
,\ldots,\tilde{b}_{m})$ be the coefficients in equation (\ref{a_matrix}).
Making use of equations (\ref{a0}) and (\ref{eqn:3.a_s}) we obtain, using the
class $\mathcal{F}_{n}^{\delta}$ in place of $\mathcal{F}_{n}$ in the
definition of $\chi^{2}_{n,k},$
\begin{align}
\label{a_tilde}\widetilde{a}_{i} &  =2\left(  a_{i}-a_{m+1}\right)
,\;i=1,\ldots,m\\
\widetilde{b}_{j} &  =2\left(  b_{j}-b_{m+1}\right)  ,\;j=1,\ldots
,m\nonumber\\
\widetilde{a}_{0} &  =2\left(  a_{m+1}+b_{m+1}\right)  .\nonumber
\end{align}

From the proof of Proposition \ref{th:matrixform} we get, setting
$\delta_{i,j}=1$ for $i=j$ and 0 otherwise,
\begin{align*}
\chi_{n,k}^{2}  & =\frac{1}{4}\underline{a}^{\prime}S_{n,k}\underline{a}\\
& = \frac{1}{4n}\sum_{i=1}^{m}\sum_{j=1}^{m}\left[ \tilde{a}_{i}\tilde{a}%
_{j}\left(  N_{i,\cdot}\delta_{i,j}-N_{i,\cdot}N_{j,\cdot}/n\right)
+\tilde{b}_{i}\tilde{b}_{j}\left(  N_{\cdot,i}\delta_{i,j} -N_{\cdot
,i}N_{\cdot,j}/n\right)  \right. \\
& \qquad\qquad\qquad\left.  +2\tilde{a}_{i}\tilde{b}_{j}\left(  N_{i,j}%
-N_{i\cdot}N_{\cdot,j}/n\right)  \right]  ,
\end{align*}
which, using (\ref{a_tilde}) and after some algebra yields
\[
n\chi_{n,k}^{2}=\sum_{i=1}^{m+1}\sum_{j=1}^{m+1}N_{i,j}(a_{i}+b_{j})^{2}.
\]

\end{proof}

We now can refine Theorem \ref{sec4:thm F infinite}.

\begin{thm}
Let (\ref{grid}) hold. Assume that $P\in\Omega$ satisfies the condition in
Lemma \ref{lm:eigen2} for some $\alpha>0$.

Let $m(n)$ be such that $\ \lim_{n\rightarrow\infty}m=\infty$ and
$\lim_{n\rightarrow\infty}m^{3/2}n^{-1/2}\log n=0.$

Then, under $H0$,
\[
\frac{n\chi^{2}_{n,k}-2m}{\sqrt{4m}}\longrightarrow N(0,1).
\]

\end{thm}

\begin{proof}
It is enough to prove $\frac{n \chi^{2}_{n,k}-\underline{\gamma}_{n,k}%
^{\prime}S_{k}^{-1}\underline{\gamma}_{n,k}}{\sqrt{4m}}=o_{P}(1).$

Denote $\hat{P}$ and $\overline{P}$ the minimizers of (\ref{chi_n_BRW}) and
(\ref{chi_n_bar}) in $\mathcal{Q}$. Let $\hat{p}_{i,j}$ and $\overline
{p}_{i,j}$ denote the respective probabilities of cells.

We write
\begin{align*}
n {\chi}^{2}_{n,k}-\underline{\gamma}_{n,k}^{\prime}S_{k}^{-1}\underline
{\gamma}_{n,k}  & \leq\sum_{i=1}^{m+1}\sum_{j=1}^{m+1}\left(  N_{i,j}%
-n\overline{p}_{i,j}\right)  ^{2}\left(  \frac{1}{N_{i,j}}-\frac{1}{np_{i,j}%
}\right) \\
& \leq\max_{i,j}\left(  \frac{np_{i,j}}{N_{i,j}}-1\right)  n\underline{\gamma
}_{n,k}^{\prime}S_{k}^{-1}\underline{\gamma}_{n,k}%
\end{align*}
and
\begin{align*}
n{\chi}_{n,k}^{2}-\underline{\gamma}_{n,k}^{\prime}S_{k}^{-1}\underline
{\gamma}_{n,k}  & \geq\min_{Q \in\mathcal{Q}}\sum_{i=1}^{m+1}\sum_{j=1}%
^{m+1}\frac{\left(  N_{i,j}-nq_{i,j}\right)  ^{2}}{np_{i,j}}\left(
\frac{np_{i,j}}{N_{i,j}}-1\right) \\
& \geq-\max_{i,j}\left| \frac{np_{i,j}}{N_{i,j}}-1\right|  \underline{\gamma
}_{n,k}^{\prime}S_{k}^{-1}\underline{\gamma}_{n,k}.
\end{align*}
Whenever
\begin{equation}
\sqrt{m}\max_{i,j}\left|  \frac{np_{i,j}}{N_{i,j}}-1\right|  \overset
{P}{\rightarrow}0\label{limit celle}%
\end{equation}
holds, then the above inequalities yield $\frac{n\chi^{2}_{n,k}-\underline
{\gamma}_{n,k}^{\prime}S_{k}^{-1}\underline{\gamma}_{n,k}}{\sqrt{4m}}%
=o_{P}\left(  \frac{\underline{\gamma}_{n,k}^{\prime}S_{k}^{-1}\underline
{\gamma}_{n,k}}{m}\right)  =o_{P}\left(  \frac{\underline{\gamma}%
_{n,k}^{\prime}S_{k}^{-1}\underline{\gamma}_{n,k}-2m}{\sqrt{4m}}\frac{2}%
{\sqrt{m}}+1\right)  =o_{P}(1) $ , which proves the claim.\newline We now
prove (\ref{limit celle}). We proceed as in Lemma 2 in \cite{BRW91}, using
inequalities (10.3.2) in \cite{ShoWell1986}. Let $B_{n}\sim Bin(n,p)$. Then,
for $t>1$,
\begin{equation}
\label{eqn:SW}\Pr\left(  \frac{np}{B_{n}}\geq t\right)  \leq\exp\left\{
-np\,h\left( 1/ t\right)  \right\}  \quad\mbox{and}\quad\Pr\left(  \frac
{B_{n}}{np}\geq t\right)  \leq\exp\left\{  -np\,h\left(  t\right)  \right\} ,
\end{equation}
where $h\left(  t\right)  =t\log t-t+1$ is a positive function.

Since $N_{i,j}\sim Bin(n,p_{i,j})$,
\begin{align*}
Pr\left\{  \max_{i,j}\left(  \frac{np_{i,j}}{N_{i,j}}-1\right)  \geq\frac
{t}{\sqrt{m}}\right\}   & \leq\sum_{i=1}^{m+1}\sum_{j=1}^{m+1}Pr\left\{
\frac{np_{i,j}}{N_{i,j}}\geq\frac{t}{\sqrt{m}}+1\right\} \\
& \leq\sum_{i=1}^{m+1}\sum_{j=1}^{m+1}\exp\left\{  -np_{i,j}h\left(  1/\left(
1+t/\sqrt{m}\right) \right)  \right\} \\
(\mbox{by (\ref{grid}) and by} p_{i,j}>\alpha p_{i,\cdot}p_{\cdot,j}) &
\leq(m+1)^{2}\exp\left\{ -c\alpha\frac{n}{\log n}(\log n)m^{-2}h\left(
1/\left( 1+t/\sqrt{m}\right) \right)  \right\} .
\end{align*}

For $x=1+\varepsilon$, $h(x)=O(\varepsilon^{2})$. Therefore, using (\ref{B})
with $k=2m$, for every $M>0$ there exists $n$ large enough that
\[
\alpha c\frac{n}{\log n}m^{-2}\,h\left(  1+\frac{-t/\sqrt{m}}{1+t/\sqrt{m}%
}\right)  \geq M
\]
and consequently $Pr\left\{  \max_{i,j}\left(  \frac{np_{i,j}}{N_{i,j}%
}-1\right)  \geq\frac{t}{\sqrt{m}}\right\} $ goes to 0.

To get convergence to zero of $Pr\left\{  \max_{i,j}\left(  1-\frac{np_{i,j}%
}{N_{i,j}}\right)  \geq\frac{t}{\sqrt{m}}\right\} =$ $Pr\left\{  \max_{i,j}
\frac{N_{i,j}}{np_{i,j}} \geq\frac{1}{1-t/\sqrt{m}}\right\} $, the second
inequality in (\ref{eqn:SW}) is used in a similar way.
\end{proof}

\begin{rem}
\label{rmk:d>2} The preceding arguments carry over to the case $d>2$ and yield
to the condition $\lim_{n}m^{d+1/2}n^{-1/2}\log n=0$. However for $d\geq6$
this ultimate condition is stronger than (\ref{D}).
\end{rem}

\section{Application: a contamination model}

Let $\mathcal{P}_{\theta}$ an identifiable class of densities on $\mathbb{R}
$. A contamination model typically writes
\begin{equation}
p(x)=(1-\lambda)f_{\theta}(x)+\lambda r(x)\label{eqn:5.contamination model}%
\end{equation}
where $\lambda$ is supposed to be close to zero and $r(x)$ is a density on
$\mathbb{R}$ which represents the distribution of the contaminating data.

An example is when $f_{\theta}(x)=\theta e^{-\theta x}$, $x>0$ and $r(x)$ is a
Pareto type distribution, say
\begin{equation}
\label{eqn:5.pareto_type}r(x):=r_{\gamma,\nu}(x)=\gamma\nu^{\gamma}
(x)^{-(\gamma+1)},
\end{equation}
with $x>\nu$ and $\gamma>1,\;\nu>1$.

Such a case corresponds to a proportion $\lambda$ of outliers generated by the
density $r_{\gamma,\nu}$.

We test contamination when we have at hand a sample $X_{1},\ldots,X_{n}$ of
i.i.d. r.v.'s with unknown density function $p(x)$ as in
(\ref{eqn:5.contamination model}). We state the test paradigm as follows.

Let $H0$ denote the composite null hypothesis $\lambda=0$, i.e.
\begin{align*}
H0  & :\;p(x)=f_{\theta_{0} }(x),\;\theta_{0} \in\Theta\\
& \mbox{versus}\\
H1  & :\;p(x)=(1-\lambda)f_{\theta}(x)+\lambda r(x)\\
& \qquad\mbox{for some }\theta\in\Theta\,\mbox{ and with }\lambda\not =0.
\end{align*}

Such problems have been addressed in the recent literature; see
\cite{lemdaniPons99} and references therein. We assume identifiability,
stating that, under $H1$, $\lambda$, $\theta$ and $r$ are uniquely defined.
This assumption holds for example when $f_{\theta}(x)=\theta e^{-\theta x}$
and $r(x)$ is like in (\ref{eqn:5.pareto_type}).

For test problems pertaining to $\lambda$ we embed $p(x)$ in the class of
density functions of signed measures with total mass 1, allowing to belong to
$\Lambda_{0}$ an open interval that contains 0.

In order to present the test statistic, we first consider a simplified version
of the problem above.

Assume that $\theta_{0} =\alpha$ is fixed, i.e. $\Theta=\left\{
\alpha\right\} $. We consider the hypotheses
\begin{align*}
H0  & :\;p(x)=f_{\alpha}(x)\\
& \mbox{versus}\\
H1  & :\;p(x)=(1-\lambda)f_{\alpha}(x)+\lambda r(x) ,\qquad\mbox{with }\lambda
\not =0.
\end{align*}

In this case $\Omega=\left\{  f_{\alpha}\right\}  $ and the null hypothesis H0
is simple.

For this problem the $\chi^{2}$ approach appears legitimate. From the
discussion in Section \ref{sec:1.intro} the $\chi^{2}$ criterion is robust
against inliers. A contamination model as (\ref{eqn:5.contamination model})
captures the outlier contamination through the density $r$. As such the test
statistic does not need to have any robustness property against those, since
they are included in the model. At the contrary, missing data might lead to
advocate in favour of $H1$ unduly. Therefore the test statistic should be
robust versus such cases.

By the necessary inclusion $f^{\ast}=2\left(  \frac{q^{\ast}}{p}-1\right)
\in\mathcal{F}$ we define
\begin{equation}
\mathcal{F}=\mathcal{F}_{\alpha}=\left\{  g=2\left(  \frac{f_{\alpha}%
}{(1-\lambda)f_{\alpha}+\lambda r}-1\right)  \text{ \ \ such that \ \ }%
\;\int|g|f_{\alpha}<\infty,\, \lambda\in\Lambda_{0} \right\}
.\label{en;5.F_alpha}%
\end{equation}

Following (\ref{eqn:2.chi_n_minimax})
\begin{equation}
\chi_{n}^{2}(f_{\alpha},p)=\sup_{g\in\mathcal{F}_{\alpha}}\int gf_{\alpha
}-T(g,P_{n}).\label{eqn:5.chi_n_alpha}%
\end{equation}

\begin{ex}
\label{ex:5.example} Consider the case $f_{\alpha}(x)=\alpha e^{-\alpha x}$
and $r(x)=\gamma\nu^{\gamma} (x)^{-(\gamma+1)}$ for some $\gamma$ fixed,
$x>\nu$.

Then $\mathcal{F}_{\alpha}= \left\{  2\left( \frac{\alpha e^{-\alpha x}%
}{(1-\lambda)\alpha e^{-\alpha x}+\lambda r_{\gamma,\nu}}-1\right)
,\:\lambda\in\Lambda_{0} \text{ such that }\int\frac{\alpha^{2}e^{-2\alpha x}%
}{(1-\lambda)\alpha e^{-\alpha x}+\lambda r_{\gamma,\nu}(x)}dx<\infty\right\}
.$
\end{ex}

Let us now turn back to composite hypothesis.

Let $\Omega$ be defined by
\[
\Omega=\left\{  q(x)=f_{\alpha}(x),\alpha\in\Theta\right\} .
\]

We can write
\[
\mathcal{F}_{\alpha}=\left\{  g(\theta,\lambda,\alpha)=2\!\!\left( \!
\frac{f_{\alpha}}{(1-\lambda)f_{\theta}+\lambda r}-1 \!\right) :
\int|g|f_{\alpha}<\infty, \lambda\in\Lambda_{0},\theta\in\Theta\right\}
\]
and
\[
\chi^{2}(\Omega,P)=\inf_{\alpha\in\Theta}\sup_{g\in\mathcal{F}_{\alpha}}\int
gf_{\alpha}-T(g,P).
\]
The supremum is to be found over a class of functions $\mathcal{F}_{\alpha}$
which changes with $\alpha$.

Denote $\Delta_{\alpha}$ the subset of $(\Theta,\Lambda_{0})$ which
parametrizes $\mathcal{F}_{\alpha}$.

\begin{ex}
[Continued]\label{ex:5.continuation1} We assume $\Theta=[\underline{\alpha
},\overline{\alpha}]$, which corresponds, in our example, to the restriction
of the expected value of $P$ (under $H0$) to the finite interval $[\frac
{1}{\overline{\alpha}},\frac{1}{\underline{\alpha}}]$.

Therefore
\begin{equation}
\label{eqn:5.chi_n_composite_1}\chi^{2}_{n}(\Omega,P)=\inf_{\underline{\alpha
}\leq\alpha\leq\overline{\alpha}} \sup_{(\theta,\lambda)\in\Delta_{\alpha}%
}\int2\left( \frac{\alpha e^{-\alpha x}}{(1-\lambda)\theta e^{-\theta
x}+\lambda r_{\gamma}(x)}-1\right) \alpha e^{-\alpha x}dx-T(g(\theta
,\lambda,\alpha);P_{n}).
\end{equation}

The supremum in (\ref{eqn:5.chi_n_composite_1}) is evaluated over a set which
changes with $\alpha$.

In accordance with the discussion in Section \ref{sec:2.estimator} we may
define
\begin{equation}
\label{eqn:5.chi_n_composite_2}%
\begin{array}
[c]{lcl}%
\mathcal{F} \!\! & \!\!=\!\! & \!\! \left\{  g(\theta,\lambda,\beta)\!=\! 2
\!\! \left( \frac{\beta e^{-\beta x}}{(1-\lambda)\theta e^{\theta x}+\lambda
r(x)}-1 \right) \!: \int\frac{\alpha\beta e^{-(\alpha+\beta)x}}{(1-\lambda
)\theta e^{-\theta x}+\lambda r} dx<\infty, (\alpha,\theta,\beta)\!\in\!
\Theta^{3}, \lambda\!\in\! \Lambda_{0} \! \right\} \\
\!\! & \!\!\subseteq\!\! & \!\! \left\{ g(\theta,\lambda,\beta): \lambda
\in\Lambda_{0},\, (\theta,\beta)\in\Gamma\!\right\} ,
\end{array}
\end{equation}
a class not depending upon $\alpha$.

The resulting test statistic would be then
\begin{equation}
\chi_{n}^{2}(\Omega,P)=\inf_{\alpha\in\Theta}\sup_{(\theta,\beta)\in
\Gamma,\;\lambda\in\Lambda_{0}}\int g(\theta,\lambda,\beta)\alpha e^{-\alpha
x}dx-T(g(\theta,\lambda,\beta);P_{n})\label{eqn:5.chi_n_composite_3}%
\end{equation}
and the supremum in (\ref{eqn:5.chi_n_composite_3}) is determined on a set
that does not depend on $\alpha$.
\end{ex}

The use of (\ref{eqn:5.chi_n_composite_1}) is proposed by M. Broniatowski and
A. Keziou \cite{Broniatowski-Keziou2003}. Also in our context it is easy to
see that (\ref{eqn:5.chi_n_composite_1}) is preferable to
(\ref{eqn:5.chi_n_composite_3}), in the sense that it reduces considerably the
computational complexity of the problem, from a subset of $\left\{
(\theta,\lambda,\beta)\in\Theta\times\Lambda_{0}\times\Theta\right\} $ to a
subset of $\{(\lambda,\theta)\in\Lambda_{0}\times\Theta\}$. \medskip

We first derive the asymptotic distribution of the test statistic $\chi
_{n}^{2}$ under $H1$; in order to use Theorem \ref{th:2.weakconv_H1} we
commute the $inf$ and $sup$ operators in (\ref{eqn:5.chi_n_composite_1})
through the following Lemma \ref{th:5.minimax}.

Assume

\begin{itemize}
\item[(A1)] $\Theta$ is compact.

\item[(A2)] For all $\alpha$ in $\Theta$, $\Delta_{\alpha}$ is compact.
\end{itemize}

Condition (A2) is verified in our example due to the compactness of the
interval $\Theta$ and to the distribution of the outliers.

\begin{lem}
\label{th:5.minimax} Let
\[
\Theta_{_{\!\!(\!\theta_{1}\!,\lambda\!,\theta_{2}\!)}}=\left\{  \alpha
\in\Theta:\,(\theta_{1},\lambda,\theta_{2})\in\Delta_{\alpha}\right\}  .
\]
Under (A1) and (A2),
\begin{align}
& \inf_{\alpha\in\Theta}\sup_{(\theta_{1},\lambda,\theta_{2})\in\Delta
_{\alpha}}\int g(\theta_{1},\lambda,\theta_{2})f_{\alpha}-T(g(\theta
_{1},\lambda,\theta_{2});P)\label{eqn:5.minimax}\\
& \quad=\sup_{(\theta_{1},\theta_{2},\lambda)\in\Theta^{2}\times\Lambda_{0}%
}\inf_{\alpha\in\Theta_{_{\!\!(\!\theta_{1}\!,\!\lambda\!,\!\theta_{2}\!)}}%
}\int g(\theta_{1},\lambda,\theta_{2})f_{\alpha}-T(g(\theta_{1},\lambda
,\theta_{2});P).\nonumber
\end{align}

\end{lem}

\begin{proof}
For $\Theta_{_{\!\!(\!\theta_{1}\!,\!\lambda\!,\!\theta_{2}\!)}}$ defined as
above we have
\begin{align}
& \inf_{\alpha\in\Theta}\sup_{(\theta_{1} ,\lambda,\theta_{2} )\in
\Delta_{\alpha} }\int g(\theta_{1} ,\lambda,\theta_{2} )f_{\alpha}%
-T(g(\theta_{1},\lambda,\theta_{2} );P)\label{eqn:5.minimax_proof_1}\\
& \quad\leq\sup_{(\theta_{1},\theta_{2},\lambda)\in\Theta^{2}\times\Lambda
_{0}}\inf_{\alpha\in\Theta_{_{\!\!(\!\theta_{1}\!,\!\lambda\!,\!\theta_{2}%
\!)}}}\int g(\theta_{1} ,\lambda,\theta_{2} )f_{\alpha}-T(g(\theta_{1}
,\lambda,\theta_{2} );P).\nonumber
\end{align}

On the other hand,
\begin{align*}
\sup_{\theta_{1}\!,\lambda\!,\theta_{2} }\!\int\!gf_{\alpha}\!-\!T(g;P) \!\!\!
& \!\!=\!\!\!\!\sup_{\theta_{1} \!,\lambda\!,\theta_{2} }\left\{
\int\!2\!\frac{f_{\theta_{2} }}{(1-\lambda)f_{\theta_{1} }+\lambda r}%
\frac{f_{\alpha}}{p}pdx\!-\!\int\!\!\left(  \!\frac{f_{\theta_{2} }%
}{(1-\lambda)f_{\theta_{1} }\!+\!\lambda r}\!\right)  ^{\!2}\!pdx\!+\!1
\right\} \\
& \!\!\!=\!\!\!\!\sup_{\theta_{1} ,\lambda,\theta_{2} }-\int\left(
\frac{f_{\theta_{2} }}{(1-\lambda)f_{\theta_{1} }+\lambda r}-\frac{f_{\alpha}%
}{p}\right) ^{2}pdx+\int\left(  \frac{f_{\alpha}}{p}-1\right) ^{2}pdx\\
& \!\!\!\leq\!\! \!\!\int\left(  \frac{f_{\alpha}}{p}-1\right) ^{2}%
pdx=\chi^{2}(f_{\alpha},p)
\end{align*}
and equality holds if $(\theta_{1},\theta_{2},\lambda)$ are such that
$\frac{f_{\alpha}}{p}=\frac{f_{\theta_{2} }}{(1-\lambda)f_{\theta_{1}
}+\lambda r}$ (identifiability allows to find such $(\theta_{1},\theta
_{2},\lambda)$ for every $\alpha\in\Theta$, and for every contaminated measure
$p$).

Also we have
\begin{align*}
& \sup_{(\theta_{1} ,\lambda,\theta_{2} )}\inf_{\alpha}\int gf_{\alpha
}-T(g;P)\\
&  \qquad=\sup_{(\theta_{1} ,\lambda,\theta_{2} )}\left\{  -\int\left(
\frac{f_{\theta_{2} }}{ (1-\lambda)f_{\theta_{1} }+\lambda r}-\frac
{f_{\alpha^{\ast}}}{p}\right)  ^{2}pdx+\int\left(  \frac{f_{\alpha^{\ast}}}%
{p}-1\right)  ^{2}pdx\right\} \\
& \qquad=\chi^{2}(f_{\alpha^{*}},p),
\end{align*}
for some $\alpha^{*}$ in $\Theta_{_{\!\!(\!\theta_{1}\!,\!\lambda
\!,\!\theta_{2}\!)}}$.

We thus get
\[
\sup_{(\theta_{1},\lambda,\theta_{2}) }\inf_{\alpha}\int g f_{\alpha
}-T(g;P)=\chi^{2}(f_{\alpha^{*}},p)\geq\chi^{2}(\Omega,p)=\inf_{\alpha}%
\sup_{(\theta_{1},\lambda,\theta_{2}) }\int g f_{\alpha}-T(g;P),
\]
which, by (\ref{eqn:5.minimax_proof_1}), concludes the proof.
\end{proof}

Theorem \ref{th:2.weakconv_H1} implies consistency of $\chi_{n}^{2}$ as an
estimator of $\chi^{2}$ and convergence in distribution of $\sqrt{n}\left(
\chi_{n}^{2}-\chi^{2}\right)  $ to a normally distributed r.v. with mean zero
and variance given by $P\left(  \left(  -g^{\ast}-\frac{1}{4}{g^{\ast}}%
^{2}\right)  ^{2}\right)  -\left(  P\left(  -g^{\ast}-\frac{1}{4}{g^{\ast}%
}^{2}\right)  \right)  ^{2}$, under $H1$.

The asymptotic distribution under the null hypothesis can be found subject to
the choice of the parametric class $\left\{  f_{\alpha}\right\}  $ and of the
density $r$, as can be deduced by Theorem 3.5 in
\cite{Broniatowski-Keziou2003}. Following their Theorem 3.5, which holds for
composite hypothesis testing in a parametric environment, the test statistic
$n\chi_{n}^{2}$ converges weakly, under $H0$, to a chi-squared distribution
with degrees of freedom depending on the dimension of the parameter space
$\Theta$ and on the cardinality of the constraints induced by $P\in\Omega$.

In the following, we focus on definition (\ref{eqn:5.chi_n_composite_1}) for
$\chi_{n}^{2}(\Omega,P)$.

The null hypothesis reduces the space $\Theta\times\Lambda_{0}$ to
$\Theta\times\left\{  0\right\} $.

Theorem 3.5 in \cite{Broniatowski-Keziou2003} implies that the degree of
freedom $d$ of the limiting chi-squared distribution equals the number of
parameters of $P$ under $H0$. In the following we assume $d=1$, as in Example
\ref{ex:5.example}.

Let $h(\theta,\lambda;x)=(1-\lambda)f_{\theta}(x)+\lambda r(x)$.

Checking conditions (C.12)-(C.15) in \cite{Broniatowski-Keziou2003} yields:

\begin{thm}
Under $H0$, with $P=P_{\theta_{0}}$, assume that

\begin{itemize}
\item[(i)] The class of contaminated densities $\left\{ h(\theta
,\lambda),\theta\in\Theta, \lambda\in\Lambda_{0}\right\} $ is $P_{\theta_{0}%
}-$identifiable;

\item[(ii)] The class of functions $\left\{ \frac{h(\alpha,\nu)}%
{h(\theta,\lambda)}, \theta\in\Theta_{\alpha}, \lambda\in\Lambda_{0},
\alpha\in\Theta, |\nu| <\varepsilon\right\} $ is $P_{\theta_{0}}-$GC for some
$\varepsilon$ small enough;

\item[(iii)] The densities $f_{\theta}$ are differentiable up to the second
order in some neighborhood $V(\theta_{0})$ of $\theta_{0}$ and $F_{\theta
}(x)=\int_{-\infty}^{x} f_{\theta}(u)du$ is differentiable with respect to
$\theta$;

\item[(iv)] There exists a neighborhood $V$ of $(\theta_{0},0,\theta_{0},0)$
such that, for every $(\theta,\lambda,\alpha,\nu) \in V$ we have
\[%
\begin{array}
[c]{lll}%
\frac{f_{\alpha}}{h(\theta,\lambda)}\leq H_{1}(x), & \quad & \frac{\ddot
{f}_{\alpha}}{h(\theta,\lambda)}\leq H_{3}(x),\\
\frac{\dot{f}_{\alpha}}{h(\theta,\lambda)}\leq H_{2}(x), &  & \frac
{r}{h(\theta,\lambda)}\leq H_{4}(x),
\end{array}
\]
\noindent where each of the functions $H_{j}$ ($j=1,2,3,4$) is square
integrable w.r. to the density $h(\alpha,\nu)$ and is in $L_{4}(P_{\theta_{0}%
})$.
\end{itemize}

Then, $n\chi_{n}^{2}$ converges to a chi-squared distributed r.v. with degree
of freedom equal to 1.
\end{thm}

\section*{Acknowledgements}

This work was supported by \emph{Progetto Ateneo di Padova} coordinated by
Prof. G. Celant.

\bibliographystyle{plain}
\bibliography{bibki2}

\begin{thebibliography}{10}

\bibitem{Aze1997}
D.~Aze.
\newblock {\em Elements D' Analyse Convexe et Variationnelle}.
\newblock Ellipses, 1997.

\bibitem{Beran77}
R.~Beran.
\newblock Minimum hellinger distance estimates for parametric models.
\newblock {\em Annals of Statistics}, 5:445--463, 1977.

\bibitem{BRW91}
P.J. Bickel, Y.~Ritov, and J.A Wellner.
\newblock Efficient estimation of linear functionals of a probability measure
  $p$ with known marginal distributions.
\newblock {\em Ann. Stat.}, 19:1316--1346, 1991.

\bibitem{BOR81}
I.S. Borisov.
\newblock On the accuracy of the approximation of empirical random fields.
\newblock {\em Th Prob. Appl.}, 26:632--633, 1981.

\bibitem{Bosq80}
D~Bosq.
\newblock Sur une classe de tests qui contient le test du chi-2.
\newblock {\em Publications de l'ISUP}, 23:1--16, 1980.

\bibitem{BM89}
J.~Bretagnolle and P.~Massart.
\newblock Hungarian construction from the nonasymptotic viewpoint.
\newblock {\em Ann. Prob.}, 17:239--256, 1989.

\bibitem{Broniatowski2002}
M.~Broniatowski.
\newblock Estimation through kullback-leibler divergence.
\newblock {\em to appear in Mathematical Methods of Statistics}, 2003.

\bibitem{BroniatowskiKeziou2003b}
M~Broniatowski and A~Keziou.
\newblock Estimation and tests for models satisfying linear constraints with
  unknown parameters.
\newblock 2003, submitted.

\bibitem{Broniatowski-Keziou2003}
Michel Broniatowski and Amor Keziou.
\newblock Parametric estimation and testing through divergences, (\textit{
  Submitted }).
\newblock 2003.

\bibitem{Csiszar1967}
I.~Csiszar.
\newblock Information type measures of difference of probability distributions
  and indirect observations.
\newblock {\em Studia Sci. Math. Hungar.}, 2:229--318, 1967.

\bibitem{MR2002302}
S.~G. Donald, G.~W. Imbens, and W.K. Newey.
\newblock Empirical likelihood estimation and consistent tests with conditional
  moment restrictions.
\newblock {\em J. Econometrics}, 117(1):55--93, 2003.

\bibitem{greenwoodNikulin}
P.~Grinwood and M.~S. Nikulin.
\newblock Some remarks with respect to the application of tests of chi-square
  type.
\newblock {\em Zap. Nauchn. Sem. Leningrad. Otdel. Mat. Inst. Steklov. (LOMI)},
  158(Probl. Teor. Veroyatn. Raspred. X):49--71, 170, 1987.

\bibitem{GOR1979}
P.~Groeneboom, J.~Oosterhooff, and F.H. Ruymgaart.
\newblock Large deviation theorem for empirical probability measures.
\newblock {\em Annals of Probability}, 7:553--586, 1979.

\bibitem{IKL93}
Kallenberg~W.C.M. Inglot, T. and T.~Ledwina.
\newblock Asymptotic behaviour of some bilinear functionals of the empirical
  process.
\newblock {\em Math. Meth. Stat.}, 2:316--336, 1993.

\bibitem{IL90}
T.~Inglot and T.~Ledwina.
\newblock On probabilities of excessive deviations for kolmogorov-smirnov,
  cramer-von mises and chi-square statistics.
\newblock {\em Ann. Stat.}, 18:1491--1495, 1990.

\bibitem{JimenezShao2001}
R.~Jimenez and Y.~Shao.
\newblock On robustness and efficiency of minimum divergence estimators.
\newblock {\em Test}, 10,2:241--248, 2001.

\bibitem{Kol94}
V.I. Koltchinskii.
\newblock Komlos-major-tusnady approximation for the general empirical preocess
  and haar expansions of classes of functions.
\newblock {\em J. Theoret. Prob.}, 7:73--118, 1994.

\bibitem{Lancaster}
H.~O. Lancaster.
\newblock {\em The chi-squared distribution}.
\newblock John Wiley \& Sons Inc., New York, 1969.

\bibitem{lemdaniPons99}
M~Lemdani and O~Pons.
\newblock Likelihood tests in contamination models.
\newblock {\em Bernouilli}, 5,4:705--719, 1999.

\bibitem{LV77}
F.~Liese and I.~Vajda.
\newblock {\em Convex Statistical Distances}.
\newblock BSB Teubner, Leipzig, 1987.

\bibitem{Lindsay1994}
B.G. Lindsay.
\newblock Efficiency versus robustness: The case of minimum hellinger distance
  and related methods.
\newblock {\em Annals of Statistics}, 22:1081--1114, 1994.

\bibitem{Mas89}
P.~Massart.
\newblock Strong approximation for multivariate empirical and related
  processes, via {KMT} reconstructions.
\newblock {\em Ann. Prob.}, 17:266--291, 1989.

\bibitem{MoralesPardoVajda1997}
D.~Morales, L.~Pardo, and I.~Vajda.
\newblock {Some} new statistics for testing hypotheses in parametric models.
\newblock {\em Journal of Multivariate Analysis}, 62:137--168, 1997.

\bibitem{neweySmith2003}
W~Newey and R~Smith.
\newblock Higher order properties of {GMM} and generalized empirical likelihood
  estimators.
\newblock {\em Cemmap Working Paper CWP04/03}, 2003.

\bibitem{Pollard84}
D~Pollard.
\newblock {\em Convergence of stochastic processes}.
\newblock Springer Series in Statistics. Springer-Verlag, New York, 1984.

\bibitem{ShoWell1986}
G.~R. Shorack and J.A. Wellner.
\newblock {\em Empirical processes with applications to statistics}.
\newblock Wiley Series in Probability and Mathematical Statistics: Probability
  and Mathematical Statistics. John Wiley \& Sons Inc., New York, 1986.

\bibitem{TV93}
M.~Teboulle and I.~Vajda.
\newblock Convergence of best phi-entropy estimates.
\newblock {\em IEEE Trans. Inform. Theory}, 39:297--301, 1993.

\end{thebibliography}

\end{document}